\documentclass[12pt]{amsart}

\usepackage{amsmath,amsthm,amsfonts,amssymb}

\theoremstyle{plain}
\newtheorem{thm}{Theorem}[section]
\newtheorem{prop}[thm]{Proposition}
\newtheorem{lem}[thm]{Lemma}

\theoremstyle{definition}
\newtheorem{definition}[equation]{Definition}

\theoremstyle{remark}

\numberwithin{equation}{section}

\newcommand\group[1]{{\text{\bf#1}}}

\newcommand{\Z}{\mathbb{Z}}

\newcommand{\R}{\mathbb{R}}

\newcommand{\C}{\mathbb{C}}

\newcommand\Or{\group{O}}
\newcommand{\U}{\group{U}}
\newcommand{\EU}{\widetilde{\group{U}}}
\newcommand{\SU}{\group{SU}}
\newcommand\Sp{\group{Sp}}

\newcommand{\union}{\cup}

\renewcommand{\phi}{\varphi}

\newcommand{\suchthat}{\,|\,}

\newcommand{\Gal}{\mathrm{Gal}}

\DeclareMathSymbol{\sdp}{\mathbin}{AMSb}{"6F}

\newcommand\bbz{\mathbb{Z}} 

\newcommand\bbr{\mathbb{R}} 
\newcommand\bbc{\mathbb{C}}

\newcommand\bbh{\mathbb{H}}

\newcommand\ca{\mathcal}

\begin{document}

\title[Stable vector bundles over a real curve]{The moduli
space of stable vector bundles over a real algebraic curve}

\author[I. Biswas]{Indranil Biswas}

\address{School of Mathematics, Tata Institute of Fundamental Research,
Homi Bhabha Rd, Bombay 400005, India}

\email{indranil@math.tifr.res.in}

\author[J. Huisman]{Johannes Huisman}

\address{D\'epartement de Math\'ematiques, Laboratoire CNRS
UMR 6205, Universit\'e de Bretagne Occidentale, 6 avenue
Victor Le Gorgeu, CS 93837, 29238 Brest cedex 3, France}

\email{johannes.huisman@univ-brest.fr}

\author[J. Hurtubise]{Jacques Hurtubise}

\address{Department of Mathematics, McGill University,
Burnside Hall, 805 Sherbrooke St. W.
Montreal, Que. H3A 2K6, Canada}

\email{jacques.hurtubise@mcgill.ca}

\date{}

\begin{abstract}
We study the spaces of stable real and quaternionic
vector bundles on a real algebraic curve.  The basic relationship is
established with unitary representations of an extension of $\bbz/2$
by the fundamental group. By comparison with the space of real or
quaternionic connections, some of the basic topological invariants of
these spaces are calculated.

\end{abstract}

\maketitle

\section{Introduction}

The moduli space of vector bundles on a
compact Riemann surface can be understood
from many points of view. Weil, in an early paper \cite{We}, launched
their study in terms of ``matrix divisors". The next major work
was the classification of vector bundles over elliptic curves
by Atiyah \cite{At}. In the middle of the
nineteen sixties, Narasimhan and Seshadri \cite{NS}, picking up on
some ideas of Weil and Atiyah, showed in a fundamental paper how the 
moduli space of stable vector
bundles on Riemann surface could be understood in terms of
unitary representations of the fundamental group of the surface. This
point of view was vastly expanded upon by Atiyah and Bott \cite{AB} in
1982, who showed how the moduli space was naturally realised inside
the space of all connections as the minima of the Yang--Mills
functional; they then used this idea to analyse the topology of the
moduli space in Morse theoretical terms, showing that the Morse function of
energy on the connections was perfect as an equivariant Morse
function, and that in consequence one could obtain the cohomology of
the moduli space in terms of the equivariant cohomology of the space of
connections (a computation that involves fairly standard ingredients)
and the cohomology of higher order critical points, which can be
obtained inductively from the moduli spaces of vector bundles of lower rank.

This work \cite{AB}, with its ties to symplectic geometry, equivariant
Morse theory, and more generally, ideas and techniques from physics,
opened up a whole series of perspectives in the study of moduli: a
complete review would take up this whole paper, but let us mention on
the side of symplectic geometry and Morse theory the efforts of
Jeffrey and Kirwan, whose techniques gave a complete computation of
the cohomology rings (see, e.g., \cite{JK}); on the physical side, the
work of E. Verlinde \cite{V}, giving Riemann--Roch numbers for the
moduli, and Witten \cite{Wi}, who gave formulae for symplectic
of theirvolumes. For these physical insights, complete mathematical proofs 
 have occupied a large number of mathematicians.

 From the point of view of representation theory, one natural 
question 
once one knows the representation ring is to compute the real and quaternionic 
representations. In our context, this means that we should restrict our attention to real algebraic curves, and study the real and quaternionic moduli on these spaces. Our point of view on real geometry follows that advocated by, e.g., Atiyah \cite{A}, in which real or quaternionic objects are complex objects invariant under an anti--holomorphic involution.

Our purpose in this paper is to consider some of the foundations of this theory. We 
begin by  settling some rather basic questions such as the topological classification of real and quaternionic bundles: equivariant topology, as usual, reserves a few surprises. We then examine the appropriate group whose representations we will study, the orbifold group of the surface under the real structure, and define the real and quaternionic representations of this group.  These representations are shown to correspond to flat real or quaternionic bundles, and a real and quaternionic version of the Narasimhan--Seshadri theorem is given. 

The rest of the paper is devoted to developing some understanding of the topology of the moduli space, in the spirit of Atiyah and Bott. A simple echoing of their ideas  in this context  presents difficulties: the spaces involved have torsion in their cohomology  and normal bundles are not necessarily orientable, for example.  We do describe the spaces of real and quaternionic connections, and compute some of their invariants such as the fundamental groups and second homotopy groups: as we have a lower  bound on the indices of the higher order critical points, this gives us the same information for the moduli spaces. 

Much thus remains to be done. The real and quaternionic fixed point
spaces are fixed point sets under an antisymplectic involution, and
the whole arsenal of symplectic techniques used to compute cohomology
has mod $2$ variants over the fixed point set; see e.g. the work of
Duistermaat \cite{Du} and that of Biss, Guillemin and Holm \cite{BGH}.
The mod 2 cohomology should then be computable following the ideas of
Atiyah and Bott, Jeffrey and Kirwan. In a similar vein, Ho and Jeffrey
\cite{HJ} have a computation of volumes on representation spaces on a
non--orientable surface which should be adaptable to this context.

The second and third authors would like to thank the T.I.F.R. 
(Mumbai) for its warm hospitality during the preparation of this 
paper. The first author wishes to thank the Harish--Chandra Research
Institute for hospitality.
The third author also would like to acknowledge the support of NSERC.

\section{Real curves, and real and quaternionic vector bundles}

Let~$Y$ be a smooth geometrically connected projective real algebraic
curve of genus $g$. The Riemann surface associated to $Y$ is the
connected complex curve
$$
X\,=\, Y(\C)\,=\, Y\times_{\mathbb R}\mathbb C
$$ obtained by
field extension. Since $Y$ is a real 
curve, the Galois group $\Gal(\C/\R)\,=\, \Z/2$
acts on $Y(\C)$; this action is an anti--holomorphic involution
$$
\sigma\,:\, Y(\C)\,\longrightarrow\, Y(\C)\, .
$$
In particular, $\sigma$ reverses orientation.
The real points $Y(\mathbb R)$ of $Y$ are the fixed 
points of $\sigma$. It follows that the
quotient $Y(\C)/\Gal(\C/\R) \,=\, X/\sigma$ is an
unoriented, and possibly
nonorientable, compact connected topological surface with, possibly
empty, boundary.

Our development uses this point of view of a real 
curve as a complex curve equipped with an involution, and from now on, 
we will let $X$ denote an irreducible complex curve equipped with an 
anti--holomorphic
involution $\sigma$. Let $X/\sigma$ denote the quotient. The
fixed points of $\sigma$ will be denoted by $X(\R)$. 

We will distinguish three types of real curve:

The
curve~$X$ is said to be of
\begin{description}
\item[Type 0]\, if~$X(\R)=\emptyset$, in which case $X/\sigma$ is a
  surface without boundary and is not orientable,
\item[Type I]\, if~$X(\R)$ is non-empty, and $X\setminus X(\R)$ is not connected, in which
case~$X/\sigma$ has a nonempty boundary and is orientable, and
\item[Type II]\, if~$X(\R)$ is non-empty, and $X\setminus X(\R)$ is connected, in which
case~$X/\sigma$ has a nonempty boundary and is not orientable.
\end{description}

Details of the cell structure of these surfaces are given below.

\textbf{Type 0.}

We consider first curves $X$ of even genus $g\,=\, 
2{\widehat g}$. The   
quotient $ X 
/\sigma$ can then be obtained by taking a surface $S$ of genus 
${\widehat g}$, 
cutting out a disk to obtain a surface with boundary $S_0$, then 
glueing the boundary circle to itself by identifying antipodal points. 
The full surface $X $ is obtained by doubling the surface $S_0$ to a 
disjoint union of $S_0$ and $\sigma(S_0)$ (with the opposite orientation) and glueing along the 
boundary by sending any point $z$ of $S_0$ in
one copy of the surface to the antipodal point of $z$ in the
other copy.
Let us choose the base point $x_0$ on the boundary of $S_0$; 
its antipodal point $\sigma(x_0)$ also lies on the boundary; let 
$\gamma$ denote a half circle on  the boundary joining these two points. 
One 
has  standard generators $\alpha_i, \beta_i$, where ${ i= 1,\dots ,{\widehat 
g}},$ and 
$(\sigma(\gamma)\circ\gamma)$ for the fundamental group of $S_0$, 
satisfying
\begin{equation}
(\prod_{i=1}^{\widehat g}[\alpha_i\, ,\beta_i])( 
\sigma(\gamma)\circ\gamma)\,=\, 1\, ;\label{boundary1}
\end{equation}
likewise there are generators  
$\alpha_i, \beta_i, { i= {\widehat g}+1,\dots ,2{\widehat g}},$ and 
$(\sigma(\gamma)\circ\gamma)$ 
for the fundamental group of $\sigma(S_0)$, satisfying 
$$(\prod_{i={\widehat g}+1}^{2{\widehat g}}[\alpha_i\, ,\beta_i])( 
\sigma(\gamma)\circ\gamma)^{-1}\,=\, 1\, .$$
The fundamental group of the whole curve $X$ then has
$\alpha_i, \beta_i$, where $i= 1,\dots ,2{\widehat g}$, as 
standard generators.

Corresponding to this description of the fundamental group, one has a 
cell decomposition of $S_0$, with two zero--cells $\{x_0\, , \sigma(x_0)\}$, 
one--cells $\alpha_i,\beta_i, i=1,...,g, \gamma,\sigma(\gamma)$, and a single 
two--cell, glued to the one--cells via the relation (\ref{boundary1}); 
this lifts to a cell decomposition of $X$ with one--cells $\alpha_i, 
\sigma(\alpha_i),\beta_i,\sigma(\beta_i),\gamma,\sigma(\gamma)$, and two 
two--cells interchanged by $\sigma$.

In a similar fashion, we can consider curves of odd genus $g= 2{\widehat 
g}+1$. One can construct the quotient $ X /\sigma$ of a surface of genus 
$g=2{\widehat g}+1$ by taking a Riemann surface $S$ of genus ${\widehat 
g}$, removing two  
disks to obtain a surface with boundary $S_0$, and glueing the two 
boundaries using an orientation preserving diffeomorphism
(as in the construction of the 
Klein bottle from a cylinder). The surface $X$ in 
turn is constructed by taking two copies $S^1_0$ and $S^2_0= \sigma(S^1_0)$ of
$S_0$ with opposite orientations  and gluing one component of the boundary of $S^1_0$ to the
other component of the boundary of $S^2_0$ using the orientation 
preserving diffeomorphism. The 
involution $\sigma$ interchanges the two boundary circles of any
copy of $S_0$, and it takes any interior point in one component to
the corresponding point in the other component.
Let $\gamma$ be a boundary circle of a copy of $S_0$. 
Choose a base point $x_0$ on $\gamma$, and join it in $S_0$ to its image 
$\sigma(x_0)$ on $\sigma(\gamma)$ by a curve $\delta$. One then has that the 
fundamental group of $S_0$ is generated by $\{\alpha_i, 
\beta_i\}_{i=1}^{\widehat g},
\gamma , \delta^{-1}\circ \sigma(\gamma)\circ 
\delta$ with the relation 
$$
(\prod_{i=1}^{\widehat g}[\alpha_i,\beta_i])\gamma  (\delta^{-1}\circ 
\sigma(\gamma)\circ \delta) = 1 .
$$
Likewise, one can choose generators $\{\alpha_i,
\beta_i\}_{{\widehat g}+2}^{2{\widehat g}+1}
,\ \gamma , \delta^{-1}\circ 
\sigma(\gamma)\circ \delta$ for the 
fundamental group of $\sigma(S_0)$, such that setting $\alpha_{{\widehat g}+1} = 
\sigma(\delta)\circ\delta, \beta_{{\widehat g}+1} = \gamma $, we have 
that  $\alpha_i, 
\beta_i$, where $i= 1,\dots ,2{\widehat g}+1$, is a standard set of 
generators 
for 
the 
fundamental group of $X$.
Corresponding to this description of the fundamental group, one has a cell 
decomposition of $S_0$, with two zero--cells $x_0, 
\sigma(x_0)$, one--cells $\{\alpha_i,\beta_i\}_{i= 1}^{\widehat g}, 
\gamma 
, \sigma(\gamma),\delta$, and a single two--cell; this lifts to a cell decomposition of 
$X$ with one--cells $\alpha_i, 
\sigma(\alpha_i),\beta_i,\sigma(\beta_i),\gamma,\sigma(\gamma),\delta, \sigma(\delta)$, 
and two two--cells interchanged by $\sigma$.

 In both cases, using the cellular decompositions as given above, the  quotient surface $X/\sigma$ is    a cofibration
\begin{equation}\label{cofibration}
\vee_{i=1}^{g+1} S^1\longrightarrow X/\sigma\longrightarrow S^2\, .
\end{equation}

\textbf{Type I.}

One now has a quotient surface $X/\sigma$ which is an orientable surface $S_0$ with 
$r$ boundary circles and genus 
$\widehat g = (1+g-r)/2$. Choose a base point on each boundary curve; 
let 
$\gamma_i$, $i = 1,\ldots ,r$, denote the boundary circles, and 
let $\delta_i$, 
$i=2,\ldots ,r$, be paths joining the base point on $\gamma_1$ to the 
base 
point on $\gamma_i$. The surface $X/\sigma$ is the one obtained by 
glueing a disk to the one skeleton $\alpha_1,\beta_1,\ldots , 
\alpha_{\widehat g},\beta_{\widehat g}, \gamma_1,\ldots , \gamma_r, 
\delta_2,\ldots , 
\delta_r$ by the boundary circle to 
$\alpha_1\beta_1\alpha_1^{-1}\beta_1^{-1}\cdots\alpha_{\widehat g} 
\beta_{\widehat g}\alpha_{\widehat g}^{-1}\beta_{\widehat g}^{-1} 
\gamma_1\delta_2\gamma_2\delta^{-1}_2\cdots \delta_r\gamma_r\delta^{-1}_r $. 
This gives a cell decomposition for $X/\sigma$ with $r$ zero--cells, 
one--cells  $\alpha_1,\beta_1,\cdots , \alpha_{\widehat 
g},\beta_{\widehat g}, 
\gamma_1,\cdots , \gamma_r, \delta_2,\cdots, \delta_r$, and one 
two--cell. 
The cofibration for $X/\sigma$ is  
\begin{equation}\label{cofibration3}
\vee_{i=1}^{g+r} S^1\longrightarrow X/\sigma\longrightarrow S^2\, ,
\end{equation}

\textbf{Type II.}

One now has a quotient surface $X/\sigma$ which is an unoriented surface 
with $r$ boundary circles. It can be obtained from an oriented surface 
$S$ of genus ${\widehat g} \,=\, (g-r)/2$ from which one removes $(r+1)$ 
open 
disks obtaining a surface $S_0$ with boundaries 
$\gamma_0,\cdots,\gamma_r$; the quotient $X/\sigma$ is 
obtained from
$S_0$ by glueing $\gamma_0$ to itself, identifying each point to its 
antipodal point. Choose a base point on each boundary curve of $S_0$; let  
$\delta_i, i=1,\dots,r$, be paths joining the base point on $\gamma_0$ 
to 
the base point on $\gamma_i$. The surface $X/\sigma$ is the one 
obtained by glueing a disk to the one--skeleton spanned by 
$\alpha_1,\beta_1,\cdots, 
\alpha_{\widehat g},\beta_{\widehat g}, \gamma_0,\cdots, \gamma_r, 
\delta_1,\cdots, \delta_r$ by the boundary circle to 
$\alpha_1\beta_1\alpha_1^{-1}\beta_1^{-1}\cdots\alpha_{\widehat g} 
\beta_{\widehat g}\alpha_{\widehat g}^{-1}\beta_{\widehat g}^{-1} 
\gamma_0\gamma_0\delta_1\gamma_1\delta^{-1}_1\cdots 
\delta_r\gamma_r\delta^{-1}_r $. This gives a cell decomposition for 
$X/\sigma$ with $r+1$ zero cells, one--cells  $\alpha_1,\beta_1,\cdots, 
\alpha_{\widehat g},\beta_{\widehat g}, \gamma_0,\cdots, \gamma_r, 
\delta_1,\cdots, 
\delta_r$, and one two--cell.

 The cofibration is 
\begin{equation}\label{cofibration4}
\vee_{i=1}^{g+r+1} S^1\longrightarrow X/\sigma\longrightarrow S^2\, ,
\end{equation}

\section{Real and quaternionic bundles}

Let $\pi:E\rightarrow X$ be a holomorphic vector bundle over $X$. Let $\overline{E}$ be the same real vector bundle, with the conjugate complex structure. Note that
$\sigma^*\overline{E}$ is also a holomorphic vector bundle
over $X$. By a \textit{lift} of $\sigma$ to $E$ we mean
an  anti-holomorphic isomorphism, antilinear on the fibers 
\begin{equation}\label{e1}
\widetilde{\sigma}\, :\, E\, \longrightarrow\,
 {E}\,,
\end{equation}
making the diagram 
\begin{equation}\begin{matrix}
 E\ \  &\buildrel{\widetilde{\sigma}}\over{ \longrightarrow}&
 {E}\ \ \\
\downarrow \pi&&\downarrow \pi\\
 X\ \ &\buildrel{ {\sigma}}\over{ \longrightarrow}&
X\ \ \end{matrix}
\end{equation}
commute. Note that this is equivalent to a holomorphic isomorphism of vector bundles
$E\simeq \sigma^*(\overline{E})$.
We will be studying real and quaternionic bundles over
real curves:

\begin{definition}
A pair $(E\, ,\widetilde{\sigma})$ as in \eqref{e1} is said to be 
{\it real} if the composition
$$
E\,\stackrel{\widetilde{\sigma}}{\longrightarrow}\,
{E}
\,\stackrel{\widetilde{\sigma}}{\longrightarrow}
\,  E
$$
is the identity map of $E$.

A pair $(E\, ,\widetilde{\sigma})$ as in \eqref{e1} is said to be 
{\it quaternionic} if the composition
$$
E\,\stackrel{\widetilde{\sigma}}{\longrightarrow}\,
{E}
\,\stackrel{\widetilde{\sigma}}{\longrightarrow}
\,  E
$$
coincides with $-\text{Id}_E$.
\end{definition}

A \textit{simple} vector bundle $E$ is one such that there are no
holomorphic endomorphisms $E\, \longrightarrow\, E$ apart from the constant scalar
multiplications.

\begin{prop} Let $E\longrightarrow X$ be simple of positive rank. 
Suppose that $E \,\sim\, \sigma^*\overline{E}$; then $E$ admits  either a real or 
or a quaternionic structure, but $E$ cannot have both. 
\end{prop}

\begin{proof}
Take an anti-holomorphic isomorphism
$$
\widetilde{\sigma} \, :\, E\, \longrightarrow \, {E}\, ,
$$
lifting $\sigma$. Since $\widetilde{\sigma}^2\,:\, E\,\longrightarrow\, E$,
$$
\widetilde{\sigma}^2\, =\, c\cdot \text{Id}_E\, ,
$$
where $c\, \in\, {\mathbb C}^*$. Since $\widetilde{\sigma}^2$
commutes with $\widetilde{\sigma}$, it follows that
$c\, \in\, {\mathbb R}\setminus\{0\}$. Note that if
$\widetilde{\sigma}$ is replaced by $a\cdot\widetilde\sigma$,
where $a\, \in\, {\mathbb R}\setminus\{0\}$, then
$(\widetilde\sigma)^2$ gets replaced by $a^2(\widetilde\sigma)^2$. Therefore,
$E$ admits either a real structure or a quaternionic structure. 

If one has a real structure
$\widetilde\sigma$ and a quaternionic structure $\widehat\sigma$ on $E$,
then $\widetilde\sigma^2\,=\, 1$, $\widehat\sigma^2 \,=\,-1$, and 
$\widetilde{\sigma}\circ\widehat{\sigma} \,= \,c$ 
for some constant $c$. Therefore,
$$
-1 = \widetilde{\sigma}^2\widehat{\sigma}^2 = 
\widetilde{\sigma}(\widetilde{\sigma}\widehat{\sigma})\widehat\sigma = 
\widetilde{\sigma}(c\cdot)\widehat\sigma = \overline{c} 
\widetilde{\sigma} 
\widehat{\sigma} = \overline{c} c\, ,
$$ 
a contradiction.
\end{proof}

One can easily find a decomposable bundle which is 
both real and quaternionic. For example, take $E\,=\,
{\mathcal O}_X\oplus {\mathcal O}_X$. Let $\sigma_0$ be
the real structure on ${\mathcal O}_X$ defined
by $f\, \longmapsto\, \overline{f\circ \sigma}$. So
$E$ has the real structure $\sigma_0\oplus\sigma_0$.
On the other hand, for any $J\, \in\, \text{GL}(2,{\mathbb C})$
with $J^2\,=\, -\text{Id}$, composing
$\sigma_0\oplus\sigma_0$ with $J$
we get a quaternionic structure on $E$.

\begin{prop}
Consider two triples $(E\, ,h\, , \widehat{\sigma})$ and
$(E'\, ,h'\, , \widehat{\sigma}')$, where
\begin{itemize}
\item $(E\, ,h)$ and $(E'\, ,h')$ are flat holomorphic hermitian vector 
bundles over $X$ of rank $n$,

\item $\widehat{\sigma}$ and $\widehat{\sigma}'$ are either both real
structures or both quaternionic structures, and

\item $\widehat{\sigma}$ and $\widehat{\sigma}'$ are unitary
with respect to $h$ and $h'$ respectively.
\end{itemize}
Assume that the
holomorphic vector bundle $E$ is isomorphic to $E'$. Then
there is a holomorphic isomorphism
$$
f\, :\, E\, \longrightarrow\, E'
$$
that transports $h$ to $h'$ and $\widehat{\sigma}$ to
$\widehat{\sigma}'$.
\end{prop}

\begin{proof}
We will first show that there is a holomorphic isomorphism
$$
f\, :\, E\, \longrightarrow\, E'
$$
that transports $\widehat{\sigma}$ to $\widehat{\sigma}'$.
For this, consider the complex vector space
$$
{\mathcal V}\, :=\, H^0(X,\, Hom(E\, ,E'))
\,=\, H^0(X,\, E'\otimes E^*)\, . 
$$
We have a conjugate linear isomorphism
$$
\phi\, :\, {\mathcal V}\, \longrightarrow\, {\mathcal V}
$$
defined by $T\, \longmapsto\, \widehat{\sigma}'\circ T\circ
\widehat{\sigma}$. Let
$$
{\mathcal V}^\phi\, \subset\,
{\mathcal V}
$$
be the real subspace fixed by $\phi$.
Since both $\widehat{\sigma}$ and
$\widehat{\sigma}'$ are involutions, it follows that
$\phi$ is also an involution. Therefore, the natural map
$$
{\mathcal V}^\phi\oplus \sqrt{-1}{\mathcal V}^\phi
\, \longrightarrow\, {\mathcal V}
$$
is an isomorphism; the automorphism $\phi$ acts on
$\sqrt{-1}{\mathcal V}^\phi$ as multiplication by $-1$.

Since $E$ is holomorphically isomorphic to $E'$, there is a nonempty
Zariski open subset of ${\mathcal V}$ consisting of isomorphisms from
$E$ to $E'$. Any nonempty Zariski open subset of ${\mathcal V}$ must
intersect ${\mathcal V}^\phi$. Any isomorphism
$E\, \longrightarrow\, E'$ lying in ${\mathcal V}^\phi$ takes
$\widehat{\sigma}$ to $\widehat{\sigma}'$.

This reduces us to the situation of a bundle $E$, with real structure $\hat \sigma$, and two Hermitian metrics $h$, $h'$, with $\sigma$ unitary for both of them. Since the Hermitian metrics $h$,$h'$ on $E$   give  rise to   flat connections, $E$ is polystable of degree zero, that is a sum $\oplus_j V_j$ of stable bundles, with all summands of degree zero; let us write this as  $\oplus_\alpha V_\alpha^{\oplus n_\alpha}$, with the $V_\alpha$ distinct bundles. Each summand $V_\alpha^{\oplus n_\alpha}$ is orthogonal to the others with respect to both $h$ and $h'$; on these summands, the two metrics are related by an element of $Gl(n_\alpha)$. Let
$V\, \subset\, E$ be a subbundle of $E$ of the form $V= \oplus_\alpha V_\alpha^{\oplus k_\alpha}$, then, for either of our two metrics $h$,$h'$:
\begin{itemize}
\item $V$ is preserved by the Chern connections associated to the metric,

\item the orthogonal complement $V^\perp = \oplus_\alpha V_\alpha^{\oplus (n_\alpha- k_\alpha)}$ is also preserved, and

\item since the map $\widehat \sigma$ is unitary, 
if $V$ is preserved by $\widehat \sigma$, then
$V^\perp$ is also preserved by $\widehat \sigma$.
\end{itemize}

These facts allow us to proceed inductively: it will be enough to prove the
proposition under the assumption that 
$(E\, , \widehat{\sigma})$ 
is irreducible; this means that the only holomorphic
subbundles of $E$  preserved by
$\widehat{\sigma}$ are
the zero subbundle and $E$ (respectively, the zero subbundle
and $E'$). 

If $(E\, , \widehat{\sigma})$ is irreducible, then there
are exactly two possibilities:
\begin{itemize}
\item $H^0(X,\, Hom(E\, ,E))\, =\, \mathbb C$, in particular,
$E$ is indecomposable.

\item $H^0(X,\, Hom(E\, ,E))\, =\, {\mathbb C}\oplus{\mathbb C}$;
in this case, $E\,=\, F\oplus\sigma^*\overline{F}$,
$$
H^0(X,\, Hom(F\, ,F))\, =\, \mathbb C\, ,
$$
and $F\, \not=\, \sigma^*\overline{F}$.
\end{itemize}
In the second case, where ${\mathcal V}\, :=\, H^0(X,\, Hom(E\, ,E))\, 
=\, {\mathbb 
C}\oplus{\mathbb C}$, the hermitian metrics on
$$
E\,=\, F\oplus\sigma^*\overline{F}
$$
is induced by   hermitian metrics on $F$. Also, the real
subspace
$$
{\mathcal V}^\phi\, \subset\, {\mathcal V}\,=\,
{\mathbb C}\oplus{\mathbb C}
$$
coincides with all homomorphisms of the form
$(\lambda\, ,\overline{\lambda})$, where $\lambda\,\in\,
{\mathbb C}$.

In both cases, scaling a holomorphic isomorphism from $E$ to $E'$
by a suitable real number, the required isomorphism is obtained.
\end{proof}

\section{The topology of real and quaternionic bundles}

As above, let $X$ be a Riemann surface of genus $g$ with an
anti--holomorphic involution $\sigma$. We are interested in the 
classification of holomorphic
vector bundles with {\it real} or {\it quaternionic} structures, i.e.,
lifts $\widetilde \sigma$ of the involution $\sigma$ to anti--holomorphic 
maps on the bundles, antilinear on the fibers, satisfying 
$\widetilde\sigma^2= 1$ for the real structures, and $\widetilde 
\sigma^2=-1$ 
for the quaternionic structures.

A first step lies in understanding how these bundles are classified 
topologically. Rank $n$ vector bundles on any manifold are classified 
by homotopy classes of maps into the classifying space $B\U_n$. This
last space is a Grassmannian of $n$--planes in infinite dimensional 
complex space, obtained as a limit of the Grassmannians of $n$--planes 
in $N$--dimensional space. We note that these Grassmannians, and the
universal bundles over them, can be given real and quaternionic 
structures, by considering the natural actions of the involutions on 
$\bbc^N$:
\begin{itemize}
\item{\it Real case}:\ $\rho(x_1,\cdots,x_n) = (\overline{x}_1, 
\cdots,\overline{x}_N)$
\item{\it Quaternionic case}:\ $\rho(x_1,\cdots,x_{2k}) = 
(-\overline{x}_2, 
\overline{x}_1,\cdots,-\overline{x}_{2k}, \overline{x}_{2k-1})$, where 
$N=2k$.
\end{itemize}

Noting now that one can average sections ($(s+\widetilde\sigma(s))/2$) 
to 
make them real or quaternionic, the usual arguments (see Milnor and 
Stasheff \cite[Section 5]{MS}) using partitions of unity give us
continuous maps of the Riemann surfaces into the Grassmannians in such a way 
that the bundles with their real or  quaternionic structures are given 
by pull--backs; also, involution--equivariant homotopy classes of
involution--equivariant maps correspond to isomorphism classes of real 
or quaternionic bundles.

\subsection{Real case}

{\it Subcase 1: $X$ has real points}. 

Let $X(\bbr)$, the curve of points fixed by $\sigma$, have $r$ 
components; the quotient $X/\sigma$ is then a surface with boundary $X(\bbr)$, that is $r$ boundary 
circles. We have given above a cell decomposition of $X/\sigma$ with  
$0$--cells on 
the boundary, one--cells that are either on the boundary, or only have 
ends on the boundary, and with one two--cell. This lifts to a cell 
decomposition of $X$ such that the cells $C$ are either on the real 
components, or are such that the pair $C,\sigma(C)$ is distinct. In 
particular, we have two $2$--cells.

The real subspace of $B\U_n$ is simply $B\Or_n$. We want to classify 
$\sigma$--invariant maps into $B\U_n$; $X(\bbr)$ is then mapped to 
$B\Or_n$. We map the $0$--skeleton to a point, and then consider the 
images of the boundary circles. The image of each boundary circle gives 
a homotopy class in $\pi_1(B\Or_n) = \bbz/2$; this corresponds to the first 
Stiefel--Whitney class of the real $\bbr^n$--bundle over the boundary 
circle.

Having fixed these homotopy classes, we want to see what degrees of 
freedom remain. The rest of the one--skeleton can be contracted to a 
point: one simply chooses one cell $C$ in each pair $C,\sigma(C)$, 
contracts it ($B\U_n$ is simply connected), and applies the involution 
in $B\U_n$ to obtain the homotopy for $\sigma(C)$. The remaining degree 
of freedom then lies in the two two--cells $c_2,\sigma(c_2)$. The 
homotopy classes of attachings of $c_2$ with fixed boundary are 
classified by $\pi_2(B\U_n) \,=\, \bbz$; as what one does to 
$\sigma(c_2)$ 
is determined by what one does to $c_2$, we are done.

It is useful however, to be a bit more explicit. One can realise the 
$2$--dimensional class generating $H_2(B\U_n)$ corresponding to the first 
Chern class as a complex projective line $S$, whose equator, in turn, is 
a one dimensional real projective line representing the 1--dimensional 
generator of $H_1(B\Or_n)$. The map of the surface into the classifying space can be chosen to lie in $S$, as follows. We start with the one-skeleton. Recall that the real components of the curve are a certain number of circles; on some of them, say $s$ of them, 
 the Stiefel--Whitney 
classes of the subbundle $E_\bbr$ over $X(\bbr)$ consisting of elements 
fixed by $\widetilde\sigma$ is non-zero; these circles can each get mapped homeomorphically, with the same orientation, to the equator of our sphere; the other boundary circles, and the other one-cells, can simply be mapped to the base point. Once one has this, the disk $c_2$ can be homotoped while fixing the boundary ($B\U_n$ is the union of $S$ and cells of dimension at least four) to 
$s$ 
times the northern hemisphere of $S$, plus a sum of $k$ copies of $S$, 
all with a positive orientation. The disk $\sigma(c_2)$ is then mapped 
to $s$ times the southern hemisphere of $S$, plus a sum of $k$ copies of 
$S$, again all with a positive orientation (the changes of orientations 
on the surface under $\sigma$ and on $S$ under reflection in the equator 
cancel). Thus our surface $X$ is mapped to $2k+s$ times the 2--sphere 
$S$, and so:

\begin{prop}
Real rank $n$ vector bundles $E$ on a real surface $X$ with real points 
are 
classified topologically by the first Stiefel--Whitney classes
$$
w_1(E_{i,\bbr})\,\in\, H^1(X_{i,\bbr},\, \bbz/2)\,=\, \bbz/2
$$
of $E_\bbr$ over the components 
$X_{i,\bbr}$ of 
$X_\bbr$, and the  first Chern class $c_1(E)\,\in\, \bbz$, subject to 
the 
restriction $c_1(E) \equiv \sum _i w_1(E_{i,\bbr})$ mod (2). All such 
combinations of Stiefel--Whitney classes and Chern classes do occur.
\end{prop}

{\it Subcase 2: $X$ has no real points}. This case is somewhat simpler. 
Here one takes the cell decomposition of $S_0$ given in Section 2, and 
lifts it to $X$, so that one has a cell decomposition 
consisting of pairs $C,\sigma(C)$. Choosing one cell from each pair, one 
can homotope the $0$--skeleton and one--skeleton in turn to a single 
point in 
$B\Or_n$. Then the two two--cells $c_2,\sigma(c_2)$ are each homotopic 
to $k$ copies of the standard two--spheres, so that the total degree over 
all of $X$ is $2k$.

\begin{prop}
Real rank $n$ vector bundles $E$ on a real surface $X$ with no real 
points are 
classified topologically by their first Chern class $c_1(E)$,  which 
must be even. All such degrees do occur.
\end{prop}

\subsection{ Quaternionic  case}

{\it Subcase 1: Odd rank}. Here, as there is no odd--dimensional 
quaternionic vector space, the curve $X $ must have no real points.

We consider first curves $X$ of even genus $g\,=\, 2{\widehat g}$, with the cell 
decomposition given in Section 2.
Now consider mappings of this to $B\U_n$, on which the involution $\rho$ 
acts without fixed points. Choosing base points $q, \rho(q)$ on the 
equator of our standard sphere $S$ (on which $\rho$ acts as the 
antipodal map), we can assume that $x_0$ is mapped to $q$, and so 
$\sigma(x_0)$ to $\sigma(q)$. Going on to the one--skeleton, we can then 
move the cells $\gamma$, $\sigma(\gamma)$ to the equator of $S$, and 
homotope all 
the other $1$--cells to points. The two--cell $c_2$ can then be homotoped 
to a copy of the northern hemisphere of $S$, plus a certain number $k$ 
of copies of $S$; the cell $\sigma(c_2)$ becomes the southern hemisphere 
of $S$, plus the same number $k$ of copies of $S$. The total degree 
becomes $2k+1$, telling us that the first Chern of the resulting bundle 
must be odd, and that this classifies the bundles.

In turn, we consider curves of odd genus $g= 2{\widehat g}+1$, again with the cell 
decomposition given in Section 2. For our mappings of $X$ to $B\U_n$, we 
can again assume that $x_0$ is mapped to $q$, and so $\sigma(x_0)$ to 
$\sigma(q)$. The one--skeleton, apart from $\delta$, $\sigma(\delta)$ now 
contracts completely to the base points. One then has that the image of 
the closure of the union of $\delta$ and the $2$--cell $c_2$, spans a 
multiple $k$ of the standard two--sphere in $B\U_n$, and as before, the 
same must be true of their image under $\sigma$. This forces the total 
degree now to be even, and so the first Chern of the resulting bundle 
must be even, and this classifies the bundles.

{\it Subcase 2: Even rank $n=2k$}. In this case, one can have fixed 
points for the real structure; they get mapped under the classifying map 
to the fixed point set of $\rho$, which is $B\Sp_k$ in $B\U_n$. The 
classification now is quite simple: taking any of the cell 
decompositions given above, we can homotope their one--skeletons to a 
single point in $B\Sp_k$, and then the two two--cells each give a 
multiple $k$ of the standard sphere $S$; this gives degree $2k$ in 
total, so that the degree of the bundle must be even.

Summarising:
\begin{prop}
The quaternionic bundles of rank $n$ over a real Riemann surface of
genus $g$ are classified by their degree $k$, with all degrees 
satisfying $k +n(g-1) \,\equiv\, 0$ mod $(2)$ occurring.
\end{prop}

An alternative proof of the necessity of the condition when the bundle is 
holomorphic is given as follows: assume that we have a quaternionic 
bundle $E$ of rank $n$, degree $k$; if one tensors this bundle by the 
real line bundle $L = d({\mathcal O}(p+\sigma(p)))$ of degree $2d$, then 
for high enough $d$ the first cohomology group of $E\otimes L$ vanishes, 
and the dimension of the zeroth cohomology is given by the Riemann--Roch 
formula: $h^0 = k + n(2d) + n(-g+1)$. The bundle $E\otimes L$ is still 
quaternionic, so the space of sections has a quaternionic structure and 
must be of even dimension.

\section{The fundamental group, its unitary representations, and flat bundles}

\subsection{Fundamental groups}

If the involution $\sigma$ is free, the  fundamental group of $X$ and 
the fundamental group of the quotient $X/\sigma$ fit into an exact sequence:
\begin{equation}\label{eqpi1}
 0\longrightarrow 
\pi_1(X)\longrightarrow\pi_1(X/\sigma)\longrightarrow\Z/2\longrightarrow 0,
\end{equation}

If the involution $\sigma$ has fixed points, the above no longer holds; 
for example, if one takes the Riemann sphere with the involution 
induced by conjugation, the quotient is a disk.
Whether $\sigma$ does have fixed points or not, however, the 
surface~$X/\sigma$ has a natural structure
of a $\Z/2$--orbifold, and we can still have the diagram \eqref{eqpi1}, 
provided we replace $\pi_1(X/\sigma)$ by the orbifold fundamental group 
$\Gamma$:
\begin{equation}\label{eqpi2}
0\longrightarrow 
\pi_1(X)\longrightarrow\Gamma\longrightarrow\Z/2\longrightarrow 0,
\end{equation}

The orbifold fundamental group of~$X/\sigma$ is also
known as the equivariant fundamental group of~$X$. Explicitly, we can 
define it as follows.

We choose a base point $x_0\, \in\, X$ such that $\sigma(x_0)\, \not=\,
x_0$. Consider the space of all homotopy classes of continuous
paths $\gamma\, :\, [0\, ,1] \, \longrightarrow\, X$, where homotopies 
fix the end points,  such that
\begin{itemize}
\item{} $\gamma(0)\, =\, x_0$, and
\item $\gamma(1)\, \in\, \{x_0\, ,\sigma(x_0)\}$.
\end{itemize}

Thus the homotopy classes split into two disjoint sets, depending on 
the image $\gamma(1)$: the first is simply the fundamental group 
$\pi_1(X) = \pi_1(X,x_0)$, and the second is the set $Path(X)= 
Path(X,x_0)$ of all homotopy classes of paths from $x_0$ to 
$\sigma(x_0)$. The disjoint union of $\pi_1(X)$ and $Path(X)$ will be 
our group $\Gamma$, with the following composition rule: let
$
\gamma_1\, , \gamma_2\, \in\,\Gamma\, .
$
If $\gamma_1(1) \, =\, x_0$, then define
$$
\gamma_2\gamma_1\, =\, \gamma_2\circ \gamma_1,
$$
where $\circ$ is the usual composition of paths. If $\gamma_1(1) \, =\,
\sigma(x_0)$, then define
$$
\gamma_2\gamma_1\, :=\, \sigma(\gamma_2)\circ \gamma_1\, .
$$
Note that for $\gamma\in Path(X)$, we have
\begin{equation}\label{inv.}
\gamma^{-1}(t)\, =\, (\sigma(\gamma))(1-t)
\end{equation}
It follows that for any two paths
$
\gamma_1\, ,\gamma_2\, \in\, Path(X)
$

\begin{equation}\label{inv2.}
\gamma^{-1}_2\gamma_1\, =\, \gamma^{-1}_2\circ\gamma_1
\end{equation}
Finally, mapping $\pi_1(X)$ to zero and $Path(X)$ to one, we have the 
exact sequence \eqref{eqpi2}, as desired.

One can give 
explicit presentations of $\Gamma$. These are computed in ~\cite{EFG}, 
which we  recall; we will also give some alternate presentations, more 
compatible with the standard ones of $X$.

{\it Case 0:  $X$ is of type~0.} The results of ~\cite{EFG} give:
$$
\Gamma=\langle\delta_1,\ldots,\delta_{g+1}\suchthat
\delta_1^2\cdots\delta_{g+1}^2=1\rangle,
$$
where~$g$ is the genus of the curve.  All the
generators~$\delta_i$ belong to $Path(X)$. 

Alternatively, one can take the standard basis for $\pi_1(X)$ produced 
in 
Section 2; the group $\Gamma$ has $\gamma$ as an extra generator. In 
the even genus ($g = 2{\widehat g}$) case:
$$
\Gamma=\langle \{\alpha_i,\beta_i\}_{i=1}^{2{\widehat g}},
\gamma\suchthat
\prod_{i-1}^{\widehat g}[\alpha_i,\beta_i])\gamma^2= 1, 
\prod_{i=1}^{2{\widehat g}}[\alpha_i,\beta_i]=1 \rangle,
$$
 and in the odd genus ($g = 2{\widehat g}+1$)case:
$$
\Gamma=\langle\{\alpha_i, \beta_i\}_{i= 1}^{2{\widehat g}+1}, 
\gamma\suchthat
 \gamma^2= \alpha_{{\widehat g}+1},  \prod_{i=1}^{2{\widehat g}+1}[\alpha_i,\beta_i]=1 \rangle,
$$

{\it Case I: ~$X$ is of type~I}. Let~$r$ be the number of connected
components of~$X(\R)$. Equivalently, $r$ is the number of boundary
components of~$X/\sigma$. Let~$\widehat g$ be the genus of~$X/\sigma$. 
Since the
``double'' of~$X/\sigma$ is a compact connected topological surface
without boundary of genus~$g$, one has~$2{\widehat g}+r\,=\,g+1$.  The 
orbifold
fundamental group $\Gamma$ is generated by
$\delta_1,\ldots,\delta_{r+{\widehat g}},\eta_1,\ldots,\eta_{r+{\widehat 
g}}$ subject to
the relations (~\cite{EFG})
\begin{gather*}
\eta_i^2=1 \text{ and } [\delta_i,\eta_i]=1, \text{ for
  $i=1,\ldots,r$}, \text{ and}\\
\delta_1\cdots\delta_r\cdot[\delta_{r+1},\eta_{r+1}]\cdots
[\delta_{r+{\widehat g}},\eta_{r+{\widehat g}}]=1.
\end{gather*}
The generators~$\eta_1,\ldots,\eta_r$ belong to $Path(X)$. The other
generators are contained in the subgroup~$\pi_1(X)$.

Alternately, one can choose the base point $x_0$ on the fixed point 
set, and let $\gamma$ be the constant path from $x_0$ to $\sigma(x_0)$. 
Then, in $\Gamma=\pi_1(X)\sqcup Path(X)$, $\gamma^2=1$ and 
$\gamma\alpha = \sigma(\alpha)\gamma$ for any element $\alpha$ of 
$\pi_1(X)$. The generators are then given by standard generators for 
the fundamental group $\pi_1(X)$, plus $\gamma$:
$$
\Gamma=\langle\{\alpha_i, \beta_i\}_{i= 1}^g, \gamma\suchthat
 \gamma^2= 1,  \prod_{i=1}^{g}[\alpha_i,\beta_i]=1, 
\gamma\alpha_i = 
\sigma(\alpha_i)\gamma, \gamma\beta_i = \sigma(\beta_i)\gamma\rangle,
$$

{\it Case II: $X$ is of type~II}. Let~$r$ be the number of connected
components of~$X(\R)$. Again, $r$ is also equal to the number of
boundary components of~$X/\sigma$. Let $k$ be the genus 
of~$X/\sigma$.  It
is understood here that the genus of a nonorientable surfaces is equal
to the number of cross--caps. In any case, it is easily seen
that $k+r=g+1$.  Then the orbifold fundamental group $\Gamma$ 
is
generated by
$\delta_1,\ldots,\delta_r,\eta_1,\ldots,\eta_{r+k}$
subject to the relations
\begin{gather*}
\eta_i^2=1 \text{ and } [\delta_i,\eta_i]=1, \text{ for
  $i=1,\ldots,r$}, \text{ and}\\
\delta_1\cdots\delta_r\cdot\eta_{r+1}^2\cdots\eta_{r+k}^2=1.
\end{gather*}
The generators~$\eta_1,\ldots,\eta_{r+k}$ belong to 
$Path(X)$. The other
generators are contained in the subgroup~$\pi_1(X)$(~\cite{EFG}).

Alternately, one can give the same description as in case I:
$$
\Gamma=\langle\{\alpha_i, \beta_i\}_{i= 1}^{g}, \gamma\suchthat
 \gamma^2= 1,  \prod_{i=1}^{g}[\alpha_i,\beta_i]=1, 
\gamma\alpha_i = 
\sigma(\alpha_i)\gamma, \gamma\beta_i = \sigma(\beta_i)\gamma\rangle,
$$
Note that these are not quite presentations, as we have not specified 
the action of $\sigma$.

\subsection{Real and quaternionic unitary representations}

We are interested in  unitary representations of
$\pi_1(X, x_0)$ which extend to $\Gamma$ in an appropriate way.

Let the Galois group~$\Gal(\C/\R)=\Z/2$ act naturally on the unitary
group $\U_n$. Such an action gives rise to a semi--direct product
of~$\U_n$ and $\Gal(\C/\R)$, which we will denote by $\EU_n$. We
call~$\EU_n$ the \emph{extended} unitary group. It acts naturally
on~$\C^n$ by complex linear and conjugate linear unitary
automorphisms.  If we denote complex conjugation
$(z_1\, ,\dots\, ,z_n)\, \longrightarrow\,
(\overline{z_1}\, ,\dots\, ,\overline{z_n})$
of ${\mathbb C}^n$ by $\tau$, then $\EU_n$
decomposes into two components
$$
\EU_n=\U_n\union\U_n\tau\, .
$$
Since~$\U_n$ is a normal subgroup in~$\EU_n$, one has a short exact
sequence
\begin{equation}\label{eqEU}
e\,\longrightarrow\, \U_n\,\longrightarrow\,\EU_n\,
\stackrel{\xi}{\longrightarrow}\, \Z/2\, \longrightarrow\, e\, ,
\end{equation}
 
\begin{definition}

Let $\rho\in {\rm Hom}(\pi_1(X),\U_n)$ be a representation of 
$\pi_1(X)$.

We say that the representation $\rho$ has a  {\it real extension  
$\widehat\rho$}\, if it extends as a representation $\widehat\rho$ from 
$\pi_1(X)$ to $\Gamma$ in such a way that the diagram
$$
\begin{matrix}
e & \longrightarrow &\pi_1(X)&\longrightarrow &
\Gamma & \longrightarrow &
{\mathbb Z}/2{\mathbb Z} & \longrightarrow & e\\
&&\Big\downarrow{\rho} &&\,~\Big\downarrow{\widehat\rho}&& \Vert\\
e & \longrightarrow & \text{U}(n) &\longrightarrow &
\widetilde{\text{U}}(n) & \longrightarrow &
{\mathbb Z}/2{\mathbb Z} & \longrightarrow & e
\end{matrix}
$$
is commutative. The extension $\widehat\rho$ will be referred to as a 
{\it real unitary representation}.

We say that the representation $\rho$ has a {\it quaternionic extension 
$\widehat\rho$}\, if it extends as a map from $\Gamma$ to 
$\widetilde{\text{U}}(n)$ in such a way that the diagram
$$
\begin{matrix}
e & \longrightarrow &\pi_1(X)&\longrightarrow &
\Gamma & \stackrel{\xi}{\longrightarrow} &
{\mathbb Z}/2 & \longrightarrow & e\\
&&\Big\downarrow{\rho} &&\,~\Big\downarrow{\hat\rho}&& \Vert\\
e & \longrightarrow & \text{U}(n) &\longrightarrow &
\widetilde{\text{U}}(n) & \longrightarrow &
{\mathbb Z}/2  & \longrightarrow & e
\end{matrix}
$$
commutes, and that $\rho(a)\rho(b)= (-1)^{\xi(a)\xi(b)} \rho(ab)$, where 
one takes now ${\mathbb Z}/2$ to be the set $\{0\, ,1\}$. The extension 
$\widehat\rho$ will be referred to as a {\it quaternionic unitary 
representation}.

\end{definition}

\begin{prop}\label{lem1}
For two extensions $\widehat\rho$ and $\widetilde\rho$, consider for $x\in Path(X)$ the unitary matrix $C(x) 
= \widehat{\rho}(x)^{-1}\widetilde{\rho}(x)$.
This $C(x)= C$ is independent of $x$ in $Path(X)$. 
Furthermore, $C$ commutes with the representation
$\rho$ on $\pi_1(X,x_0)$.

If $\rho\, :\, \pi_1(X,x_0)\, \longrightarrow\, {\rm U}(n)$ is 
irreducible, it cannot simultaneously have a real and a
quaternionic extension.

\end{prop} 

\begin{proof}For $x, y\in Path(X)$, we have
$$\widehat{\rho}(x)^{-1}\widetilde{\rho}(x)(\widehat{\rho}(y)^{-1}
\widetilde{\rho}(y))^{-1}= 
\widehat{\rho}(x)^{-1}\widetilde{\rho}(xy^{-1})\widehat{\rho}(y)= 
\widehat{\rho}(x)^{-1}\widehat{\rho}(xy^{-1})\widehat{\rho}(y)=1
$$
(note that in the quaternionic case, $\widehat{\rho}(x)^{-1} = 
-\widehat{\rho}(x^{-1})$).
We have, for $z$ in $\pi_1(X)$, that 
$$
\widehat{\rho}(x)^{-1}\widetilde{\rho}(x)\rho(z)(\widehat{\rho}(x)^{-1}
\widetilde{\rho}(x))^{-1} \,=\,
\widehat{\rho}(x)^{-1}\widetilde{\rho}(xzx^{-1})\widehat{\rho}(x)
$$
$$
=\, \widehat{\rho}(x)^{-1}\widehat{\rho}(xzx^{-1})\widehat{\rho}(x)
\,=\, \widehat{\rho}(z)=\rho(z)\, ,
$$ 
so that $C$ commutes with the representation.

Let $\widehat\rho$ be real and let $\widetilde\rho$ be quaternionic. 
For $x\in Path(X)$, we have $x^2\,\in\, \pi_1(X,x_0)$. Hence
if $\widehat{\rho}(x)^{-1}\circ \widetilde{\rho}(x)\, =\,
C\cdot \text{Id}$,
$$
1 = \widehat{\rho}(x^{-2})\widetilde{\rho}(x^2)= 
-\widehat{\rho}(x)^{-1}(C\cdot) \widetilde{\rho}(x)) =
-\overline{C}C\, ,
$$
a contradiction.
\end{proof}

\subsection{Unitary representations and flat bundles}

Now let 
$$\text{Rep}_\bbr(X, n)$$
denote the family of real representations of $\Gamma$, and 
$$\text{Rep}_\bbh(X, n)$$
denote the family of quaternionic representations of $\Gamma$.

Two homomorphisms
$\rho\, ,\rho'\, \in\, \text{Rep}_{\bbr}(X, n)$ (or both
in $\text{Rep}_{\bbh}(X, n)$)
are called \textit{equivalent} if there is an element 
$A\, \in\, \text{U}_n$ such that
$$
\rho'(g)\, =\, A\rho(g)A^{-1}, 
$$
for all $g\, \in\, \Gamma$.
The set of equivalence classes of elements of $\text{Rep}_{\bbr}(X, n)$, 
(respectively, $\text{Rep}_{\bbh}(X, n)$)
will be denoted by $\widetilde{\rm Rep}_{\bbr}(X, n)$ (respectively, 
$\widetilde{\rm Rep}_{\bbh}(X, n)$).

\begin{thm}
There is a natural bijective correspondence between the
equivalence classes of elements of $\widetilde{\rm Rep}_{\bbr}(X, n)$
(respectively, $\widetilde{\rm Rep}_{\bbh}(X, n)$) and the
isomorphism classes of triples $(V, \nabla^V\, \sigma)$, where
$(V, \nabla^V)$ is a unitary flat vector bundle over
$X$ of rank $n$ and
$$
\sigma\, :\, V\, \longrightarrow\, V
$$
is a real(respectively, quaternionic) structure on $V$ such that 
$\nabla^V$ is 
$\sigma$--equivariant:
$$ \nabla^V(x) \,=\, (\widetilde\sigma^{-1})^*\nabla^V(\sigma(x))$$
\end{thm}

\begin{proof}
{\it 1. From representation to bundle}. Let us start with a 
representation $\rho$ of $\Gamma$. Its restriction to $\pi_1(X)$ 
defines in the usual fashion a flat bundle over $X$: one starts with 
the trivial flat unitary bundle $\widetilde X\times \bbc^n$ over the 
universal cover $\widetilde X$, and if $g\in \pi_1(X)$ is represented by 
a 
deck transformation, one identifies $p\times \bbc^n$ with 
$g(p)\times\bbc^n$ in such a way that the trivial flat bundle descends 
to a bundle $E$ on $X$ with monodromy representation given by $\rho$. 

If $\gamma$ is any homotopy class of  paths from $x$ to $y$, let 
$T_\gamma$ represent the parallel transport from $x$ to $y$. 
If $x$ and $y$ are both the base point $x_0$, one can identify 
$$T_\gamma = \rho(\gamma): E_{x_0}\longrightarrow E_{x_0}\, .$$
We define the lift $\widetilde\sigma$ of $\sigma$ to $E$ as follows. At 
$x_0$, we choose a path $\gamma$ from $x_0$ to $\sigma(x_0)$, and set: 
\begin{eqnarray} 
\widetilde{\sigma}_{x_0}: E_{x_0}&\longrightarrow &E_{\sigma(x_0)}\\
v&\longmapsto & T_\gamma\ \rho(\gamma)^{-1}(v); \end{eqnarray}

\begin{lem}\label{lem2}
The isomorphism $\widetilde{\sigma}_{x_0}$ is independent of the choice
of the path $\gamma$ connecting $x_0$ to $\sigma(x_0)$.
\end{lem}

\begin{proof}
Take any other path $\delta\, :\, [0\, ,1]\, \longrightarrow\, X$
such that $\delta(0)\, =\, x_0$ and $\delta(1)\, =\,
\sigma(x_0)$. Consider the element
$\delta^{-1}\gamma \, =\, \delta^{-1}\circ\gamma\, \in\, 
\pi_1$.
Since $\rho(\delta^{-1}\gamma) \, =\, \rho(\delta)^{-1}
\rho(\gamma)$, and $T_{\delta^{-1}\circ\gamma} = T_{\delta^{-1}}\circ 
T_{\gamma}$, the result follows.
\end{proof}

For  other $x\in X$, choose any path $\delta$ from $x$ to $x_0$, and set
\begin{eqnarray} 
\widetilde{\sigma}_x: E_{x}&\longrightarrow &E_{\sigma(x)}\\
 v&\longmapsto &T_{\sigma(\delta)}^{-1}\widetilde{\sigma}_{x_0}T_\delta   
\end{eqnarray}

\begin{lem}\label{lem3}
The isomorphism $\widetilde\sigma_{x}$ is independent of the choice
of the path $\delta$ connecting   $x$ to $x_0$.
\end{lem}
\begin{proof}
If $\delta, \delta'$ are two paths, one wants to show that 
$$\widetilde\sigma_{x_0} = T_{\sigma(\delta')} 
T_{\sigma(\delta)}^{-1}\widetilde\sigma_{x_0}T_\delta T_{\delta'}^{-1}$$
This means that we want
$$\widetilde\sigma_{x_0} = T_{\sigma(\delta')} T_{\sigma(\delta)}^{-1} 
T_\gamma \rho(\gamma)^{-1}\rho(\delta \delta'^{-1})$$
This however is precisely the definition of $\widetilde\sigma_{x_0}$ 
using the path 
$\sigma(\delta')\circ\sigma(\delta)^{-1}\circ\gamma$
\end{proof}

\begin{lem}\label{invtparallel} Parallel transport is invariant under 
$\widetilde\sigma$: if $\eta$ joins $x$ to $y$, then 
$$ T_{\sigma(\eta)} = 
\widetilde\sigma_{y}T_\eta\widetilde\sigma_x^{-1}\label{sigmapartrans}$$
\end{lem}

This is just a consequence of the definition of $\widetilde\sigma_{y}$ 
by 
parallel transport from the base point.

\begin{lem}
$\widetilde{\sigma}_{\sigma(x)} \widetilde{\sigma}_x = 1$, if the 
representation 
is real, and $-1$, if it is quaternionic.
\end{lem}

\begin{proof} When $x=x_0$, one has
\begin{eqnarray}
\widetilde{\sigma}_{\sigma(x_0)} \widetilde{\sigma}_{x_0} &= 
T_{\sigma(\gamma^{-1})}^{-1} 
T_\gamma\rho(\gamma)^{-1}T_\gamma^{-1}T_\gamma\rho(\gamma)^{-1}\\
&= \rho(\gamma^2)\rho(\gamma)^{-1}\rho(\gamma)^{-1}\end{eqnarray}
which is +1 if the representation is real, and -1 if it is quaternionic.
The proof for other points follows by parallel transport.
\end{proof}

We have now defined a flat bundle with the right properties  from the 
representation. We now must check that the result is invariant under 
the equivalence of representations. This is straightforward.

{\it 2. From bundle to representation}

We now turn things around, and take a real (respectively, quaternionic) 
bundle 
$(E,\widetilde\sigma)$, along with a $\widetilde\sigma$--invariant 
connection 
$\nabla$. This then has invariant parallel transport, as defined in 
\eqref{invtparallel}. Choosing a basis for the fiber $E_{x_0}$ over the 
base point $x_0$, we can, as usual, define  the holonomy representation 
$\rho: \pi_1(X)\longrightarrow \U_n$:
$$\rho(g) \,=\, T_g$$
 One must then extend this to $\Gamma$. 
Inverting the procedure of part 1. of the proof gives, for a path 
$\gamma$ from $x_0$ to 
$\sigma(x_0)$:
$$ \rho(\gamma) \,=\, \widetilde{\sigma}_{x_0}^{-1} T_\gamma\, .$$
It is immediate that this is anti--linear. Composing two elements of 
$Path(X)$ gives:
\begin{eqnarray}\rho(\gamma)\rho(\gamma')&= 
\widetilde{\sigma}_{x_0}^{-1} 
T_\gamma\widetilde{\sigma}_{x_0}^{-1} T_{\gamma'}\\
&= 
\widetilde{\sigma}_{x_0}^{-1}\widetilde{\sigma}_{\sigma(x_0)}^{-1}
T_{\sigma(\gamma)} T_{\gamma'}\\
&=\widetilde{\sigma}_{x_0}^{-1}\widetilde{\sigma}_{\sigma(x_0)}^{-1} 
T_{\sigma(\gamma)\circ  \gamma'}\\ & = \widetilde{\sigma}_{x_0}^{-1}
\widetilde{\sigma}_{\sigma(x_0)}^{-1} \rho(\gamma\gamma'),\end{eqnarray}
which is $\pm \rho(\gamma\gamma')$ depending on whether the bundle is 
real or quaternionic. 

One checks in a similar fashion that the other compositions (when one 
of the two elements lies in $\pi_1(X)$) satisfy the right relations. 
Thus real bundles give real representations, and quaternionic bundles 
give  quaternionic ones. We note that the global conjugations of $\rho$ 
by a unitary matrix correspond to changes in our unitary trivialisation 
 of $E_{x_0}$; these changes give rise to equivalent 
representations. This completes the proof of the theorem.
\end{proof}

\subsection{Representations and flat connections on $S_0$.}

The correspondence of real or quaternionic representations with flat 
connections enables (or at least makes more evident) a description of these representations in terms of representations of the fundamental group of the surface $S_0$. 

\begin{itemize}
\item{} {\it Type 0, even genus; real bundles.} In this case, as we have seen, one has that $S_0$ is a once punctured surface of genus $g/2$, with two marked points chosen on the boundary which are interchanged by the real structure. We note that a trivialisation at one of the marked points gives a trivialisation at the other. Restricting flat connections to $S_0$, and noting that apart from some constraints on the boundary, connections on $S_0$ determine real connections $X$, we obtain, by integrating, for the moduli space of flat connections:
\begin{equation}{\ca M}_\bbr =  \Big\{\{A_i, B_i\}_{i= 1}^{g/2}, C\in 
\U_n| 
\prod_{i=1}^{g/2}[A_i,B_i]C\overline{C} =1\Big\}\Big/\U_n\label{rep0even}\end{equation}
 The $\U_n$ action is by
\begin{equation} (A_i, B_i, C)\longmapsto 
(gA_ig^{-1},gB_ig^{-1},gC\overline{g}^{-1})\end{equation}

\bigskip

\item{} {\it Type 0, odd genus; real bundles.} Here, $S_0$ is a twice punctured surface of genus $(g-1)/2$, with  one  marked point  chosen on each boundary, which are interchanged by the real structure. The moduli space is then 
\begin{equation}{\ca M}_\bbr =  \Big\{\{A_i, B_i\}_{i= 1}^{(g-1)/2}, 
C, D\in \U_n| 
\prod_{i=1}^{g/2}[A_i,B_i]CD\overline{C}D^{-1} =1\Big\}\Big/\U_n\label{rep0odd}\end{equation} The $\U_n$ action is by
\begin{equation} (A_i, B_i, C,D)\longmapsto (gA_ig^{-1},gB_ig^{-1},gCg^{-1}\, , 
 gD\overline{g}^{-1})\end{equation} \bigskip
 
\item{} {\it Type I; real bundles.} Here, $S_0$ is an $r$--punctured surface of genus 
$\hat g= (g-r+1)/2$, with  one  marked point  chosen on each boundary; one has 
\begin{eqnarray}{\ca M}_\bbr& = & \Big\{\{A_i, B_i\}_{i= 1}^{\widehat 
g}, \{D_j\}_{j=2}^r\in \U_n,  \{C_j\}_{j= 1}^{r},\in 
\Or_n|\nonumber \\ &&\quad\quad  
\prod_{i=1}^{\widehat g}[A_i,B_i]C_1\prod_{i=2}^{r}D_jC_jD_j^{-1}  =1\Big\}\Big/(\Or_n)^r.\label{repI}\end{eqnarray}
 The group  action is by
\begin{equation} (g_1,\cdots ,g_r)(A_i, B_i, C_j,D_j)\longmapsto 
(g_1A_ig_1^{-1},g_1B_ig_1^{-1},g_jC_jg_j^{-1}, g_1D_jg_j^{-1})\end{equation}
\bigskip

\item{} {\it Type II; real bundles.} Here, $S_0$ is an $r+1$--punctured surface of 
genus $\hat g= (g-r)/2$, with  one  marked point  chosen on each boundary; one has 
\begin{eqnarray}{\ca M}_\bbr& =  & \Big\{\{A_i, B_i\}_{i= 
1}^{\widehat g}, C_0, 
\{D_j\}_{j=1}^{r}\in \U_n,\,  \{C_j\}_{j= 1}^{r}\in 
\Or_n|\nonumber \\ &&\quad\quad 
\prod_{i=1}^{\widehat g}[A_i,B_i]C_0\overline{C}_0\prod_{i=1}^{r}D_jC_jD_j^{-1}  =1\Big\}\Big/\U_n\times (\Or_n)^r.\label{repII}\end{eqnarray}
 The group  action is by
\begin{equation} (g_0, g_1,\cdots ,g_r)(A_i, B_i, C_j,D_j)\longmapsto 
(g_0A_ig_0^{-1},g_0B_ig_0^{-1}, g_0 C_1\overline{g}_0^{-1}, g_jC_jg_j^{-1}, g_0D_jg_j^{-1})\end{equation}
\bigskip

\item{} {\it Quaternionic bundles} In all cases the formulae are the 
same, except that one replaces conjugation ($C\longmapsto \overline C$) by 
minus conjugation ($C\longmapsto -\overline C$) in the formulae, and $\Or_n$ 
by $\Sp_{n/2}$, when they occur in 
\eqref{rep0even}, \eqref{rep0odd}, \eqref{repI}, \eqref{repII}.
\end{itemize}

\subsection{Central extensions}

So far we have discussed flat bundles, corresponding to representations 
of the orbifold fundamental group; these bundles of necessity have 
trivial degree. For bundles of arbitrary degree $k$, one can consider 
projectively flat connections, or alternately connections on the 
complement of a point with central monodromy around that point which is 
given by $\exp(2\pi\sqrt{-1}k/n)$. In our case, it is in fact more 
convenient 
to choose a pair of points $p,\sigma(p)$, with $p\neq \sigma(p)$; the 
constraint on monodromy is then that the monodromy around each point be 
given by $\exp(\pi\sqrt{-1}k/n)$. (Note that if $L$ is a small loop 
around $p$, 
with monodromy $\exp(\pi\sqrt{-1}k/n)$, reality forces the monodromy 
along 
$\sigma(L)$ to be $\exp(-\pi\sqrt{-1}k/n)$; on the other hand $\sigma$ 
also 
changes orientation.)

From the point of view of representations of the fundamental group 
$\pi_1(X)$, one is taking a representation of a central extension of the 
group; real and quaternionic representations correspond to 
representations of a central extension of $\Gamma$: one has the diagram 
of central extensions
\begin{equation}
\begin{matrix}
e & \longrightarrow &{\mathbb Z}&\longrightarrow &
\pi_1(X)' & \longrightarrow &
\pi_1(X)& \longrightarrow & e\\
&&\Big\downarrow  &&\,~\Big\downarrow && \Vert\\
e & \longrightarrow & {\mathbb Z} &\longrightarrow &
\Gamma' & \longrightarrow &
\Gamma & \longrightarrow & e
\end{matrix}
\end{equation}

From the point of view of connections on $S_0$, only one of our two 
points, say $p$, lies in $S_0$.  Our moduli spaces with non--zero degree 
is then given by formulae 
\eqref{rep0even}, \eqref{rep0odd}, \eqref{repI}, \eqref{repII}, but 
with the $=1$ 
replaced by $= \exp(\pi\sqrt{-1}k/n)$.

\section{Connections and vector bundles}

\subsection{Real and quaternionic connections and the Yang--Mills functional}

We have examined, for real and quaternionic bundles, both  their topological classification and their flat or projectively flat structures. Now we want to study the moduli space of their holomorphic structures. One of the best approaches to understanding these
is that pioneered by Narasimhan and Seshadri \cite{NS}, whereby the 
bundles are defined by the $\overline\partial$--operator of a unitary 
connection. The space of these connections has a natural energy 
functional, the Yang--Mills functional, given by the $L^2$--norm of the 
curvature of the connections. The gradient flow of this functional 
preserves the holomorphic structure, and allows us to relate moduli to 
the set of critical points, following the approach of Atiyah and Bott 
\cite{AB}.

Let us first consider the case of degree zero. The theorem of Narasimhan 
and Seshadri \cite{NS}, in this context, says that minima of the 
functional, which are flat connections, correspond to stable or 
polystable bundles; integrating these connections gives representations 
of the fundamental group of the surface into the unitary group. The 
proof of this theorem given by Donaldson in \cite{Do} follows precisely 
the variational viewpoint.

Extending to bundles of non--zero degree $k$, the minima of the 
Yang--Mills functional now correspond to having a curvature tensor that 
is central, and a constant multiple of the K\"ahler form. Integrating 
the connection, however, only gives a representation into $PU_n$; 
instead, one can, following the approach of \cite{AB} and our definition of the central extension given above, modify the 
connection so that it is flat on the complement of a point, with a 
  simple central pole  and residue $2\pi\sqrt{-1} (k/N)$ at the
point; 
this amounts to finding an appropriate scalar $1$--form. Minima then 
correspond to representations of the extension of the fundamental group  
described above.

Now let us put in real structures. We restrict to trivialisations 
invariant under the real structure. The real structure $\sigma$ acts on 
connections by pull--back of forms and conjugation; note that since 
$\sigma$ is anti--holomorphic, the real structure interchanges the 
$(0,1)$ and the $(1,0)$ components. In local coordinates near a fixed 
point: $\sigma^*(A(z) dz+ A'(z) d\overline{z}) = 
\overline{A}'(\sigma(z)) dz + 
\overline{A}(\sigma(z))d\overline z$. The connection is real if it is 
mapped 
to itself under the involution. 

Our involution preserves the Yang--Mills functional, and symmetric 
criticality gives us that a real holomorphic bundle has a constant 
central curvature real connection, which is invariant under the real or quaternionic structure, if and only if it is polystable as a 
complex bundle. (From the Morse flow point of view, if one starts at an invariant connection, one has to end up at one.)   For bundles of non--zero degree, we can then modify the 
connection now by a scalar one--form so that the connection is flat and that at a pair 
of distinct  points $p,\sigma(p)$ there is a pole whose monodromy  is $\exp(\pi\sqrt{-1} 
k/n)$.

\begin{thm} {\rm (Real version of Narasimhan--Seshadri theorem.)}
The moduli space of polystable real (respectively, quaternionic) bundles 
of 
degree $k$ and rank $n$ is diffeomorphic to the space of equivalence 
classes of real (respectively, quaternionic) unitary representations of 
the extended orbifold fundamental group 
$\Gamma'$ in ${\rm U}(n)$ which map the centre to $\exp(\pi\sqrt{-1} 
k/n)$.
\end{thm}

Note that stability as a complex bundle for a real bundle is the same as being  stable as a real bundle. Indeed, 
if $E$ is destabilised by $F$, it is destabilised either by the sheaf
$F+\sigma^*F$ or by the sheaf $F\cap\sigma^*F$.

\subsection{The space of connections, and gauge groups.}

The core idea of Atiyah and Bott in developing their understanding of 
the topology of the space of stable bundles then is to think of it as 
being the space of minima of the Yang--Mills functional on connections 
modulo gauge, and think of this topology of the minima in a Morse 
theoretic way as being obtained, homologically at least, by 
``subtracting" from the topology of all connections the topology of the 
higher order critical points. The total space of connections modulo 
gauge is the quotient of the contractible space ${\ca A}$ of connections 
by the group of gauge transformations, and this group is close to acting 
freely; indeed, generically, the stabilisers consist of constant, 
central gauge transformations. Quotienting the group by these, 
therefore, our space of connections modulo gauge is almost the 
classifying space of the quotient group. Indeed, replacing the space of 
connections by the classifying space of this group allows an inductive 
computation, at least in the equivariant setting. In our case, the space ${\ca A}$ of unitary connections is 
replaced by the space ${\ca A}_\bbr$ of real ($\sigma$--invariant) 
connections; it is again an affine space, and so contractible. We again 
have a group ${\ca G}_\bbr$ of real ($\sigma$--invariant) gauge 
transformations, which, up to the subgroup $Z$ of constant central gauge 
transformations, acts generically freely; the classifying spaces of 
${\ca G}_\bbr$ and of  $\overline{\ca G}_\bbr={\ca G}_\bbr/Z$ are our objects 
of interest.

We first consider the case of real bundles. In our calculations, we will 
suppose that $n$ is greater than $2$, and give the results at the end 
for $n=1,2$.

\noindent{\it Real case, type 0 curves}

For curves $X$ of even genus $g= 2\widehat g$, we consider the cell 
decompositions of $S_0$ given in Section 2, with in particular 
one--cells 
$ \alpha_1,\beta_1,\cdots, \alpha_{\widehat g},\beta_{\widehat g}, \gamma $; on $X$, 
these cells get doubled into pairs $c, \sigma(c)$. Similarly, in the odd 
genus case, one uses the description of the surface $S_0$ given above, 
with cells $ \alpha_1,\beta_1,\cdots, \alpha_{\widehat g},\beta_{\widehat g}, \gamma, 
\delta$; again, on $X$, these cells get doubled into pairs $c, 
\sigma(c)$.

Let us now consider the $\sigma$--equivariant gauge group ${\ca G}_\bbr$ 
on $X$; we begin by considering the based gauge group ${\ca G}_\bbr^0$ 
of gauge transformations which are the identity over the base point, 
taken to lie on the boundary of $S_0$.  As a gauge group on $S_0$, ${\ca 
G}_\bbr$ corresponds to the subgroup of all unitary gauge 
transformations such that their restriction to the boundary satisfies 
$g(x) = \overline{g}(\tau(x))$, where $\tau $ is the antipodal map. For 
the based gauge group, if $x_0$ is the base point, $g(x_0) = 
g(\tau(x_0))= 
Id$. From  the cofibration above (\ref{cofibration}), and the cell 
decompositions we have given (with cells mostly coming in pairs 
$c,\sigma(c)$, so that one just has to specify the gauge transformation 
on $c$),  one has the fibration for the based gauge group:
\begin{equation}\label{fibration}
{\ca G}^0(S^2)\longrightarrow {\ca G}_\bbr^0 \longrightarrow 
\prod_{i=1}^{g+1} 
\Omega(\U_n)\, ,
\end{equation}
where ${\ca G}^0(S^2)= \Omega^2(\U_n)$ is the based gauge group on the two--sphere. For the 
based gauge group, the restriction map to the base point on $X$ gives 
the fibration
\begin{equation}\label{fibrationbase}
{\ca G}_\bbr^0\longrightarrow {\ca G}_\bbr \longrightarrow \ \U_n\, .
\end{equation} 
Finally, we will want to quotient by the (real) centre of $\U_n$:
\begin{equation}\label{fibrationproj}
\bbz/2 \longrightarrow {\ca G}_\bbr\longrightarrow \overline{\ca 
G}_\bbr\, .
\end{equation} 
This then yields the corresponding fibrations on classifying spaces (we note that $B\Omega G = G_0$, the connected component of the identity):

\begin{equation}\label{Bfibration}
\Omega (\U_n)_0\longrightarrow B{\ca G}_\bbr^0 \longrightarrow 
\prod_{i=1}^{g+1}  \U_n,
\end{equation}
\begin{equation}\label{Bfibrationbase}
B{\ca G}_\bbr^0\longrightarrow B{\ca G}_\bbr \longrightarrow \ B\U_n,
\end{equation} 
\begin{equation}\label{Bfibrationproj}
B\bbz/2\longrightarrow B{\ca G}_\bbr\longrightarrow B\overline{\ca 
G}_\bbr\, .
\end{equation} 

We now want the first few homotopy groups of these spaces. We note that 
$\pi_i(BG) = \pi_{i-1}(G)$, for any group $G$.

 For the fundamental groups, the homotopy sequences for the fibrations give:
\begin{equation}
\begin{matrix}
0&\longrightarrow&\pi_1(B{\ca 
G}_\bbr^0)&\longrightarrow&\bbz^{g+1}&\longrightarrow&0\\
\bbz&\longrightarrow&\pi_1(B{\ca G}_\bbr^0)&\longrightarrow&\pi_1(B{\ca 
G}_\bbr)&\longrightarrow&0\\
\bbz/2&\longrightarrow&\pi_1(B{\ca 
G}_\bbr)&\longrightarrow&\pi_1(\overline{B}{\ca 
G}_\bbr)&\longrightarrow&0\end{matrix}\end{equation}
The first sequence tells us that $\pi_1(B{\ca G}_\bbr^0)= \bbz^{g+1}$.
For the middle sequence we want to know what the image of $\bbz$ is.  We work instead with 
$$\pi_1(\U_n) = \bbz\longrightarrow\pi_0({\ca 
G}_\bbr^0)\longrightarrow\pi_0({\ca G}_\bbr)\longrightarrow 0$$
Let $f(t)$ represent a generator of $\pi_1(\U_n)$, with $f(0)=f(1) =1$. Choosing a disk around the base point, with a radial coordinate such that $r=1$ is the boundary of the disk, one can lift $f(t)$ to  ${\ca G}_\bbr$, by 
$F(t,r) = f((1-r)t)$, extending $F$ by the identity outside the disk, 
with the obvious exception that near the conjugate of the base point, 
one has $\overline{F}(t,r)$ . The question is then whether $F(1,r)$ lies 
in the same connected component of ${\ca G}_\bbr^0$ as the identity. To 
fix our ideas, let us look first at two cases, the case $X=S^2$  and 
$n=1$, and the case $X=T^2$, the torus,  and $n=1$. Our real structures 
are the antipodal map for the two--sphere, the fixed point--free 
structure for the torus, and complex conjugation on the circle.

\begin{lem}
The $\sigma$--equivariant based maps $g:S^2\longrightarrow S^1$ 
are classified by $\bbz$; the $\sigma$--equivariant unbased 
maps, by $\bbz/2$.

The $\sigma$--equivariant based maps $g:T^2\longrightarrow S^1$ which are 
homotopically trivial as non--equivariant maps are classified by 
$\bbz$;
the $\sigma$--equivariant, homotopically trivial as non--equivariant maps, unbased 
maps, 
by an element of $\bbz/2$.
\end{lem}
\begin{proof} An equivariant map  $g: S^2\longrightarrow S^1$ is determined 
by 
a map from the two--disk $D^2$ to $S^1$, subject to the constraint that 
the restriction to the boundary circle $\partial D^2$ is equivariant. We 
can lift any map to a map $D^2\longrightarrow \bbr$, where $\bbr$ covers 
$S^1$ in the standard way, mapping the integers to the identity. Now 
lift our map $g$ to $\hat g$, with the base point being mapped to $0\in \bbr$. 
Choosing a point $p$ on the boundary circle of $D^2$, let $a=\hat g(p)$ be the image of $p$. On the lifted map to the 
line,  $\sigma$--equivariance manifests itself by the requirement on 
$\partial D^2$ that the image of $\sigma(p)$ in $\bbr$ be $-a +m$, for 
some integer $m$. This integer will be our invariant: indeed, any 
equivariant homotopy $G$ will have a lift $\hat G$ satisfying $\hat G(t,\sigma(p)) = -\hat G(t,p) 
+m$; there have to be points in the image on both sides of $m/2$, if the 
map is non constant; on the other hand, the boundary can be homotoped to 
the constant map  to $m/2$. The mapping on the rest of the disk can 
then be homotoped to a standard map. If one drops the basing condition, 
we note that the point $a$ is only defined up to an integer $k$; this 
in turn means that $m$ is defined only up to $2k$. Thus one still has 
the parity of $m$ as an invariant, but nothing else.

The maps of the torus follow essentially the same pattern. 
\end{proof}

In the case of interest to us here, the restriction of our map $F$ to 
the   cycles $\alpha_i, \beta_i$ is homotopically trivial: the cycle 
exits, then enters the disk. One can then up to homotopy   contract 
these cycles of $X$ to a point to obtain a map $$\widehat F\,:\, 
S^2\,\longrightarrow \,\U_n$$ or a map  $\widehat F: T^2\longrightarrow 
\U_n$, 
depending 
on the parity of the genus.  One then can take the determinant, to get 
$$\widehat G\,:\, S^2\,\longrightarrow \,\U_1$$ or $\widehat G: 
T^2\longrightarrow \U_1$; the 
fact that our original 
$f$ was a generator of the fundamental group of $\U_n$ tells us that 
the relevant $m$ in the preceding lemma is $\pm2$. In short, the map 
$\bbz\longrightarrow\pi_1(B{\ca G}_\bbr^0)$ is an injection. 

One can ask what its image is.  The isomorphism $\pi_1(B{\ca G}_\bbr^0)= 
\bbz^{g+1}$ is given by restriction to the one--skeleton. As noted, on 
the cycles $\alpha_i, \beta_i$ this restriction is zero (the cycle 
exits, then reenters the disk); in the even genus case, on the cycle 
$\gamma$, it goes to twice the generator (on $X$, the cycle exits the 
disk centred at $p$ to go to the disk centred at $\sigma(p))$; in the 
odd genus case, the restriction is trivial on $\gamma$, and twice a 
generator on $\delta$. This then gives $\pi_1(B{\ca G}_\bbr)= 
\bbz^{g}\oplus\bbz/2$, in both cases.

The final sequence amounts to $\bbz/2\longrightarrow\pi_0({\ca 
G}_\bbr)\longrightarrow\pi_0(\overline{\ca G}_\bbr)\longrightarrow 0$.  
The 
image of $\bbz/2$ depends on whether one can deform $1$ to $-1$ in ${\ca 
G}_\bbr$. When $n$ is even, this is possible: setting $n= 2k$ and choosing a basis so that 
the real structure is given by $\sigma (z_1,z_2,\cdots,z_{2k-1},z_{2k}) 
= (\overline{z}_2,\overline{z}_1, \cdots , \overline{z}_{2k}, 
\overline{z}_{2k-1})$, one can define a real path by ${\rm diag} 
(e^{i\theta},e^{-i\theta},\cdots,e^{i\theta},e^{-i\theta})$. In the odd 
dimensional case, one can use the lifting to $\bbr$ argument in the 
lemma above to show that it is not. This gives $\pi_1(B\overline{\ca 
G}_\bbr)= \bbz^{g}\oplus\bbz/2$, when $n$ is even, and $\bbz^{g}$, when 
$n$ is odd.

For the second homotopy group, one then has
\begin{equation}
\begin{matrix}
\bbz&\longrightarrow&\pi_2(B{\ca G}_\bbr^0)&\longrightarrow&0&\\
0&\longrightarrow&\pi_2(B{\ca G}_\bbr^0)&\longrightarrow&\pi_2(B{\ca 
G}_\bbr)&\longrightarrow&0\\
0&\longrightarrow&\pi_2(B{\ca G}_\bbr)&\longrightarrow&\pi_2( 
B\overline{\ca 
G}_\bbr)&\longrightarrow&\bbz/2&\longrightarrow& 
0\end{matrix}\end{equation}
 One can think of these sequences 
in terms of classifications of bundles on $S^2\times X$; in these terms, 
the top row  map $\bbz\longrightarrow\pi_2(B{\ca G}_\bbr^0)$ in essence 
is 
the restriction of the second Chern class, which is non--zero. Thus 
$\pi_2(B{\ca G}_\bbr^0) = \bbz$, and so, from the second sequence, 
$\pi_2(B{\ca G}_\bbr)=\bbz$. In the last sequence, we have seen that the map to $\bbz/2$ is onto when 
$n$ is even, and zero when $n$ is odd. When $n$ is even, 
one has two classes, one related to the second Chern classes, and the 
other, to a global path from $1$ to $-1$; we saw above that the latter lifted 
to $B\overline{\ca G}_\bbr)$; thus $\pi_2(B\overline{\ca 
G}_\bbr)=\bbz\oplus \bbz/2$ when $n$ is even, $\bbz$ when $n$ is odd. 

\bigskip
\noindent{\it Real case, type I curves}

Let ${\ca G}_\bbr^0$ be the group of gauge transformations which are the 
identity at each base point, with one base point per boundary circle, as 
above.  One notes that, in the one--skeleton, along the $\gamma_i$ the 
gauge transformations must lie in the $\Or_n$ while elsewhere they lie 
in $\U_n$. In particular, note that each $\delta_i$ contributes a 
$\Omega(\U_n)$ to ${\ca G}_\bbr^0$. The corresponding fibrations are 
\begin{equation}\label{fibration3}
\Omega^2(\U_n)\longrightarrow {\ca 
G}_\bbr^0 \longrightarrow \prod_{i=1}^{g} \Omega(\U_n)\times 
\prod_{i=1}^{r} 
\Omega(\Or_n),
\end{equation}
\begin{equation}\label{fibrationbase3} 
{\ca G}_\bbr^0 \longrightarrow {\ca G}_\bbr\longrightarrow 
\prod_{i=1}^{r}\Or_n 
\end{equation}
\begin{equation}\label{fibrationproj3} \bbz/2\longrightarrow 
{\ca G}_\bbr \longrightarrow \overline{\ca G}_\bbr
\end{equation} 
\begin{equation}\label{fibrationbaseproj3} {\ca G}_\bbr^0 
\longrightarrow 
\overline{\ca G}_\bbr\longrightarrow (\prod_{i=1}^{r}\Or_n)/\pm 
\end{equation}
and so:
\begin{equation}\label{Bfibration3} 
\Omega(\U_n)_0\longrightarrow B{\ca G}_\bbr^0 \longrightarrow 
\prod_{i=1}^{g} 
\U_n\times \prod_{i=1}^{r} (\Or_n)_0\, ,
\end{equation} 
\begin{equation}\label{Bfibrationbase3} B{\ca G}_\bbr^0 \longrightarrow 
B{\ca G}_\bbr\longrightarrow \prod_{i=1}^{r}B\Or_n
\end{equation} 
\begin{equation}\label{Bfibrationproj3} B\bbz/2\longrightarrow B{\ca 
G}_\bbr 
\longrightarrow B\overline{\ca G}_\bbr
\end{equation}
For the fundamental 
groups, the homotopy sequences for the fibrations give, for $n>2$:

\begin{equation}
\begin{matrix}
0&\longrightarrow&\pi_1(B{\ca G}_\bbr^0)&\longrightarrow&\bbz^{g}\oplus 
(\bbz/2)^r&\longrightarrow&0\\
(\bbz/2)^r&\longrightarrow&\pi_1(B{\ca 
G}_\bbr^0)&\longrightarrow&\pi_1(B{\ca 
G}_\bbr)&\longrightarrow&(\bbz/2)^r\longrightarrow&0\\
(\bbz/2)&\longrightarrow&\pi_1(B{\ca 
G}_\bbr)&\longrightarrow&\pi_1(\overline{B}{\ca 
G}_\bbr)&\longrightarrow&0\end{matrix}\end{equation}
The first sequence gives $\pi_1(B{\ca G}_\bbr^0)
\,=\, \bbz^{g}\oplus (\bbz/2)^r$. Again, on the second sequence, it is 
easier to shift downwards, as for the type 0 case:
$$(\bbz/2)^r\longrightarrow\pi_0({\ca 
G}_\bbr^0)\longrightarrow\pi_0({\ca 
G}_\bbr)\longrightarrow(\bbz/2)^r\, .$$ 
The map 
$(\bbz/2)^r\longrightarrow\pi_0({\ca G}_\bbr^0)$, as constructed above, 
gives us elements of ${\ca G}_\bbr^0$ localised on disks, through which 
pass the real components of  $X$; their restriction to the real 
components is homotopically trivial  in $\pi_1(\Or_n)$; on the other 
hand, one has an isomorphism
$\pi_2(\U_n, \Or_n)\longrightarrow\pi_1(\Or_n)$, and so one can contract 
each of these elements to the constant map on $X$; in short, the 
map is trivial. Also, the map $\pi_0({\ca 
G}_\bbr)\longrightarrow(\bbz/2)^r$ is surjective; lifting the elements 
$(1,\cdots ,1,-1,1,\cdots,1)$ in $(\bbz/2)^r$ correspond to choosing  maps 
whose restriction to 
$\delta_i, i=2,\dots,r$ is a path from the identity to ${\rm diag}(-1,1,\cdots 
,1)$ in 
$\U_n$, and 
one can choose these so that their square, once restricted, is a generator of the 
fundamental group in $\U_n$, and so lie in $\pi_0({\ca 
G}_\bbr^0)$. On the other hand,  the element 
$(-1,-1,\cdots,-1)$ in $(\bbz/2)^r$ corresponds  to  the constant (over $X$) map to
$diag (-1,1,\cdots,1)$, whose square is the identity.  This tells us that 
$\pi_1(B{\ca G}_\bbr)= \pi_0({\ca G}_\bbr) = \bbz^{g}\oplus 
(\bbz/2)^{r+1}$. Finally, one can also look at the covering
(\ref{fibrationproj3}):
$$ \bbz/2\longrightarrow \pi_0({\ca G}_\bbr) \longrightarrow  
\pi_0(\overline{\ca G}_\bbr)\longrightarrow  0\, .$$
This is, as above, trivial when $n$ is even, injective when $n$ is odd, 
and so $$\pi_1(B\overline{\ca G}_\bbr)= \pi_0(\overline{\ca G}_\bbr) = 
\bbz^{g}\oplus (\bbz/2)^{r+1}$$ for $n$ even, and $\bbz^{g}\oplus (\bbz/2)^{r}$ for $n$  
odd.

For the second homotopy groups, one has
\begin{equation}
\begin{matrix}
\bbz&\longrightarrow&\pi_2(B{\ca G}_\bbr^0)&\longrightarrow&0 \\
0&\longrightarrow&\pi_2(B{\ca G}_\bbr^0)&\longrightarrow&\pi_2(B{\ca 
G}_\bbr)&\longrightarrow&(\bbz/2)^r&\longrightarrow&0\\
0&\longrightarrow&\pi_2(B{\ca G}_\bbr)&\longrightarrow &\pi_2( 
B\overline{\ca 
G}_\bbr)&\longrightarrow&\bbz/2\end{matrix}\end{equation}
In the top sequence, one has, as before, that the map 
$\bbz\longrightarrow\pi_1({\ca G}_\bbr^0)$ is injective; thus  
$\pi_1({\ca 
G}_\bbr^0) = \bbz$. This   can be thought of as being ``localised" on 
the $S^1\times S^2$ in the cofibration (\ref{cofibration3}). The second 
sequence then tells us that there are classes in $\pi_1({\ca G}_\bbr)$ 
represented by an integer, localised on $S^1\times S^2$, and classes in 
$(\bbz/2)^r$, localised on $S^2\times$(one--skeleton); more explicitly, 
one can represent the classes on the one--skeleton by a loop in $\Or_n$ along the real 
curves, which when one moves into $\U_n$, deforms to the identity. If we then take our loop on the real curve, we can extend it into a neighbourhood in such a way that it is the identity on the boundary of the neighbourhood. This tells us that the sequence splits: $\pi_2(B{\ca G}_\bbr) = \bbz\oplus (\bbz/2)^r$.  Finally, one also has:
$$0\longrightarrow \pi_1({\ca G}_\bbr^0) \longrightarrow  
\pi_1(\overline{\ca 
G}_\bbr)\longrightarrow  
\pi_1((\prod_{i=1}^{r}\Or_n)/\pm)\longrightarrow 
0\, ,$$ 
with $\pi_1((\prod_{i=1}^{r}\Or_n)/\pm)$ is $(\bbz/2)^{r+1}$ when 
$n=0(4)$, $(\bbz/2)^{r-1}\oplus\bbz/4$ when $n=2(4)$, and $(\bbz/2)^{r}$ 
when $n$  is odd. Again, the sequence splits, and so 
$\pi_2(B\overline{\ca 
G}_\bbr) = \bbz\oplus (\bbz/2)^{r+1}$ when $n=0(4)$, $\bbz\oplus(\bbz/2)^{r-1}\oplus\bbz/4$ when $n=2(4)$, and $\bbz\oplus(\bbz/2)^{r}$ when $n$  is odd.

For $n=2$, we get   for the fundamental groups $\pi_1(B{\ca G}_\bbr^0) = 
\bbz^g\oplus \bbz^r$,  $\pi_1(B{\ca G}_\bbr) = \pi_1(B\overline{\ca 
G}_\bbr) 
=\bbz^g\oplus \bbz^r\oplus\bbz/2$, and  $\pi_2(B{\ca G}_\bbr^0) = \bbz$,  
$\pi_2(B{\ca G}_\bbr) = \bbz^{r+1}$, and $\pi_2(B\overline{\ca G}_\bbr) 
= \bbz^{r+1}$.

\bigskip
\noindent{\it Real case, type II curves}

Again, we use the cell decompositions given in Section 2. Let 
${\ca G}_\bbr^0$ be the group of gauge transformations which are the identity at each base point; one 
notes that along the $\gamma_i, i= 1,\dots,r$, the gauge 
transformations 
must 
lie in the $\Or_n$ while elsewhere they lie in $\U_n$. Note that each $\delta_i$ contributes a $\Omega(\U_n)$ to ${\ca G}_\bbr^0$.
The corresponding fibrations are
\begin{equation}\label{fibration4}
\Omega^2(\U_n)\longrightarrow {\ca G}_\bbr^0 \longrightarrow 
\prod_{i=1}^{g 
+1} \Omega(\U_n)\times \prod_{i=1}^{r} \Omega(\Or_n)\, ,
\end{equation}
\begin{equation}\label{fibrationbase4}
{\ca G}_\bbr^0\longrightarrow {\ca G}_\bbr \longrightarrow \U_n\times 
\prod_{i=1}^{r}\Or_n.
\end{equation} 
\begin{equation}\label{fibrationproj4}
\bbz/2\longrightarrow {\ca G}_\bbr \longrightarrow \overline{\ca G}_\bbr
\end{equation}  
\begin{equation}\label{fibrationbaseproj4}
{\ca G}_\bbr^0 \longrightarrow  \overline{\ca G}_\bbr\longrightarrow  
(\U_n\times \prod_{i=1}^{r}\Or_n.)/\pm
\end{equation} 
This gives
\begin{equation}\label{Bfibration4}
\Omega (\U_n)_0\longrightarrow B{\ca G}_\bbr^0 \longrightarrow 
\prod_{i=1}^{2\widehat g +r+1} \U_n\times \prod_{i=1}^{r} ( \Or_n)_0\, ,
\end{equation}
\begin{equation}\label{Bfibrationbase4}
B{\ca G}_\bbr^0\longrightarrow B{\ca G}_\bbr \longrightarrow B\U_n\times 
\prod_{i=1}^{r}B\Or_n\, ,
\end{equation}
and 
\begin{equation}\label{Bfibrationproj4}
B\bbz/2\longrightarrow B{\ca G}_\bbr \longrightarrow B\overline{\ca 
G}_\bbr
\end{equation} 

Repeating the arguments for type I and type 0 curves, one obtains the 
results given in the table below.

\noindent {\it Quaternionic case, $n$ odd.}

In this case, as we have seen, the curve can have no fixed points. One can still consider invariant connections, and gauge transformations which commute with the involution $\sigma$. Again, the constant central gauge transformations which commute with $\sigma$ are $\pm1$
 We then can compute, using (\ref{cofibration}), fibrations as in 
(\ref{fibration}) , \eqref{fibrationbase},  \eqref{fibrationproj}. The 
results appear below.

\noindent{\it Quaternionic case, $n$ even, type 0 curve.}

 We get fibrations as in (\ref{fibration}), \eqref{fibrationbase}, 
\eqref{fibrationproj}, and the same results.

\noindent{\it Quaternionic case, $n$ even, type I curve.}

Here one has fibrations as in 
(\ref{fibration3}), 
(\ref{fibrationbase3}), (\ref{fibrationproj3}), 
(\ref{fibrationbaseproj3}),
but with $\Sp_{n/2}$ replacing $\Or_n$. This simplifies things, as 
$\Sp_{n/2}$ is connected and simply connected. Results are given below.

\noindent{\it Quaternionic case, $n$ even, type II curve.}

Here one has fibrations as in 
(\ref{fibration4}),
(\ref{fibrationbase4}), (\ref{fibrationproj4}),
(\ref{fibrationbaseproj4}),
but with $\Sp_{n/2}$ replacing $\Or_n$.  
Again, see below.
\bigskip
\newpage
\bigskip

\bigskip 

{\bf Summary: Homotopy groups of the classifying spaces}
\bigskip

{\footnotesize{
\begin{tabular}{|r||c|c|c||c|c|c|}\hline
&&&&&&\\
&$\pi_1(B{\ca G}_\bbr^0)$ &$\pi_1(B{\ca G}_\bbr)$& $\pi_1(B\overline{\ca 
G}_\bbr)$& $ \pi_2(B{\ca G}_\bbr^0)$ & $\pi_2(B{\ca G}_\bbr)$ 
&$\pi_2(B\overline{\ca G}_\bbr)$\\\hline\hline
Real,&$\bbz^{g+1}$&$\bbz^{g}\oplus\bbz/2$&$\bbz^{g}\oplus\bbz/2$,&$\bbz (n> 1)$&$\bbz(n>1)$&$\bbz\oplus \bbz/2$\\
 type 0 &&& ($n$ even)& && ($n$ even)\\
&&&$\bbz^{g}$($n$ odd)& &&$\bbz$ ($n$ odd, $>1$)\\
&&&&$0\ (n=1)$&$0\ (n=1)$&$0\ (n=1)$\\
\hline
Real,&$\bbz^{g}\oplus (\bbz/2)^r$&$\bbz^{g}\oplus (\bbz/2)^{r+1}$&$\bbz^{g}\oplus (\bbz/2)^{r+1}$&$\bbz$&$\bbz\oplus (\bbz/2)^r$&$\bbz\oplus (\bbz/2)^{r+1}$\\
type I&$(n>2)$&$(n>2)$& ($n>2$, even)&$(n>1)$&$(n>2)$& ($n =0(4))$\\
&&&$\bbz^{g}\oplus (\bbz/2)^{r}$& &&$\bbz\oplus(\bbz/2)^{r-1}\oplus\bbz/4$\\
&&& ($n>1$ odd)&&& ($n =2(4) ,n>2)$\\
&&&&&&$\bbz\oplus(\bbz/2)^{r}$\\
&&&&&&($n>1$ odd)\\
&$\bbz^{g+r}  $&$\bbz^{g+r}\oplus \bbz/2  $&$\bbz^{g+r}\oplus \bbz/2  $&&$\bbz^{r+1}  $&$\bbz^{r+1}  $\\
&$ (n=2)$&$ (n=2)$&$  (n=2)$&&$  (n=2)$&$ (n=2)$\\
&$\bbz^g (n=1)$&$\bbz^g\oplus\bbz/2 (n=1)$& $\bbz^g (n=1)$& $0\ (n=1)$&$0\ (n=1)$&$0\ (n=1)$\\
\hline
Real,&$\bbz^{g+1}\oplus (\bbz/2)^r$&$\bbz^{g}\oplus (\bbz/2)^{r+1}$&$\bbz^{g}\oplus (\bbz/2)^{r+1}$&$\bbz$&$\bbz\oplus (\bbz/2)^r$&$\bbz\oplus (\bbz/2)^{r+1}$\\
type II&$(n>2)$&$(n>2)$& ($n>2$, even)&$(n>1)$&$(n>2)$& ($n =0(4))$\\
&&&$\bbz^{g}\oplus (\bbz/2)^{r}$& &&$\bbz\oplus(\bbz/2)^{r-1}\oplus\bbz/4$\\
&&& ($n>1$ odd)&&& ($n =2(4) ,n>2)$\\
&&&&&&$\bbz\oplus(\bbz/2)^{r}$\\
&&&&&&($n>1$ odd)\\
&$\bbz^{g+1+r}  $&$\bbz^{g+r}\oplus \bbz/2  $&$\bbz^{g+r}\oplus \bbz/2  $&&$\bbz^{r+1}  $&$\bbz^{r+1}  $\\
&$ (n=2)$&$ (n=2)$&$  (n=2)$&&$  (n=2)$&$ (n=2)$\\
&$\bbz^{g+1} (n=1)$&$\bbz^g\oplus\bbz/2 (n=1)$& $\bbz^g (n=1)$& $0\ (n=1)$&$0\ (n=1)$&$0\ (n=1)$\\
\hline
Quat.,&$\bbz^{g+1}$&$\bbz^{g}\oplus\bbz/2$&$\bbz^{g}\oplus\bbz/2$,&$\bbz$&$\bbz$&$\bbz\oplus \bbz/2$\\
 type 0 &&& ($n$ even)&$n>1$&$n>1$& ($n$ even)\\
&&&$\bbz^{g}$($n$ odd)& &&$\bbz$ ($n>1$ odd)\\
&&&&$0\ (n=1)$&$0\ (n=1)$&$0\ (n=1)$\\
\hline
Quat.,&$\bbz^{g }$&$\bbz^{g}$&$\bbz^{g}$,&$\bbz$&$\bbz$&$\bbz\oplus \bbz/2$\\
 type I &&& &&&\\
(n even) &&& &&&\\
\hline
 Quat.,&$\bbz^{g+1 }$&$\bbz^{g}\oplus \bbz/2$&$\bbz^{g}\oplus \bbz/2$,&$\bbz$&$\bbz$&$\bbz\oplus \bbz/2$\\
 type II &&& &&&\\
(n even) &&& &&&\\
\hline \end{tabular} }}

\bigskip
\vfill\eject
Using the determinant map $\U_n\longrightarrow \U_1$ and the various 
fibrations in mapping spaces that are deduced from it, we can compute the homotopy groups in the $\SU_n$ case, or, more generally, for fixed determinant:
\bigskip

{\bf Homotopy groups of the classifying spaces, fixed determinant $(n>1)$}

\bigskip

{\footnotesize{
\begin{tabular}{|r||c|c|c||c|c|c|}\hline
&&&&&&\\
&$\pi_1(B{\ca SG}_\bbr^0)$ &$\pi_1(B{\ca SG}_\bbr)$& 
$\pi_1(B\overline{\ca 
SG}_\bbr)$& $ \pi_2(B{\ca SG}_\bbr^0)$ & $\pi_2(B{\ca SG}_\bbr)$ 
&$\pi_2(B\overline{\ca SG}_\bbr)$\\\hline\hline
Real,&$0$&$0$&$0$,&$\bbz $&$\bbz$&$\bbz\oplus \bbz/2$\\
 type 0 &&&  & && ($n$ even)\\
&&& & &&$\bbz$ ($n$ odd, $>1$)\\
\hline
Real,&$ (\bbz/2)^r$&$  (\bbz/2)^{r}$&$(\bbz/2)^{r}$&$\bbz$&$\bbz\oplus (\bbz/2)^r$&$\bbz\oplus (\bbz/2)^{r+1}$\\
type I&$(n>2)$&$(n>2)$&  & &$(n>2)$& ($n =0(4))$\\
&&&& &&$\bbz\oplus(\bbz/2)^{r-1}\oplus\bbz/4$\\
&&&  &&& ($n =2(4), n>2)$\\
&&&&&&$\bbz\oplus(\bbz/2)^{r}$\\
&&&&&&($n$ odd)\\
&$\bbz^{r} (n=2)$&$\bbz^{r} (n=2)$&$\bbz^{r}(n=2)$& &$\bbz^{r+1} (n=2)$&$\bbz^{r+1} (n=2)$\\

\hline
Real&$ (\bbz/2)^r$&$  (\bbz/2)^{r}$&$(\bbz/2)^{r}$& $\bbz$&$\bbz\oplus (\bbz/2)^r$&$\bbz\oplus (\bbz/2)^{r+1}$\\
type II&$(n>2)$&$(n>2)$&  & &$(n>2)$& ($n =0(4))$\\
&&&& &&$\bbz\oplus(\bbz/2)^{r-1}\oplus\bbz/4$\\
&&&  &&& ($n =2(4),n>2)$\\
&&&&&&$\bbz\oplus(\bbz/2)^{r}$\\
&&&&&&($n$ odd)\\
&$\bbz^{r} (n=2)$&$\bbz^{r} (n=2)$&$\bbz^{r}(n=2)$ &&$\bbz^{r+1} (n=2)$&$\bbz^{r+1} (n=2)$\\
\hline
Quat.,&$0$&$0$&$0$&$\bbz$&$\bbz$&$\bbz\oplus \bbz/2$\\
 type 0 &&& &&& ($n$ even)\\
&&& & &&$\bbz$ ($n$ odd)\\
\hline
Quat.,&$0$&$0$&$0$&$\bbz$&$\bbz$&$\bbz\oplus \bbz/2$\\
 type I &&& &&&\\
(n even) &&& &&&\\
\hline
 Quat.,&$0$&$0$&$0$&$\bbz$&$\bbz$&$\bbz\oplus \bbz/2$\\
 type II &&& &&&\\
(n even) &&& &&&\\
\hline \end{tabular}}}

\subsection{Critical points of the Yang--Mills functional}

We noted above  that the Yang--Mills functional is invariant under the real involution; 
this then tells us that the gradient flows, which take us to critical points in the 
full space of connections, are preserved by the real structure: the real gradient flow takes us to a critical point in the full space. Conversely, if there is a direction $v$ in which the derivative of the Yang--Mills functional is negative, then it is negative in $\sigma_*(v)$ and so in $v+\sigma_*(v)$. Critical points for the real flows are then real critical points for the full space, and so, applying the results of Atiyah and Bott (\cite{AB}), we have, as noted above,

\begin{prop}
The critical points of the Yang--Mills functional on the space of real (resp. quaternionic) connections on 
our surface correspond to sums of real  (resp. quaternionic)  bundles, each equipped with a connection with constant central curvature.
\end{prop}

The discussion above, and the results of Atiyah and Bott, give us a real bundle which decomposes a sum of bundles with constant central curvature; the only thing left to remark is that the real structure must respect this decomposition.

We next must compute the indices of the critical points, adapting \cite{AB}. In the full space, (see \cite{AB})critical points correspond to a sum of bundles  $\oplus_iV_i$ with possibly different slopes  (degree/rank); the index is given by the real dimension of the space $H^1(\Sigma,\oplus_{\mu(V_i)>\mu(V_j)}V_i^*\otimes V_j)$, where $\mu$ is the slope. At a real critical point, the real index will then be of dimension one half of this, i.e., its complex dimension. 

\begin{prop}
For non--minimal critical points in degree $k$, rank $N$, the index is bounded bounded 
below by $1+ (n-1) (g-1)$.

Thus, the inclusion of the space of minima of the Yang Mills functional (polystable  real bundles) into the space of connections (modulo gauge) induces isomorphisms in homology and homotopy groups in dimensions less than $(n-1) (g-1)$, and surjections in dimension $(n-1) (g-1)$.
\end{prop}
\begin{proof} Let us suppose that a real bundle $E$ of degree $k$, rank $n$ is destabilised by a bundle $F$ of degree $\ell$, rank $m$;
we have $\ell n > k m$. The index is given by $h^1(X, Hom(F,E/F))\geq m(n-m) (g-1) - 
c_1(Hom(F,E/F)) =  m(n-m) (g-1) - (-\ell(n-m) + (k-\ell)m $, by Riemann--Roch. By the stability condition, this is bounded below by $m(n-m)(g-1) + 1$, which in turn is bounded below by
$(n-1)(g-1) +1$. 
\end{proof}

Furthermore, looking at the various degrees $k$, case by case, we can improve the bound given in the proposition:
\begin{itemize}
\item{}  $(n,g)= (2,2)$: For $k=0$, the index is bounded below by 3, instead of 2.
\item{}  $(n,g)= (3,2)$: For $k=0$, the index is bounded below by 5, instead of 3.
\item{}  $(n,g)= (2,3)$: For $k=0$, the index is bounded below by 4, instead of 3.
\end{itemize}

The same estimates apply for the case of fixed determinant. 

The quotient space of real connections modulo gauge, being the quotient of an affine space, is connected. Thus, applying the information on the non-minimal critical points of the Yang-Mills functional informs us on the connectedness of the space of minima, that is the moduli space of polystable bundles:

\begin{thm}
For $g\geq 2, n\geq 2$, there are connected moduli spaces of real or quaternionic polystable
holomorphic bundles for each allowed topological type. The topological type is determined by the first Chern class, as well as Stiefel--Whitney classes of the restriction of the real bundles to $X(\bbr)$.
\end{thm}

We next note that we have computed the first two homotopy groups of a space very close to the space 
of connections, that is the spaces $B\overline{\ca 
G}_\bbr$. Indeed, if ${\ca A}$ is the affine space of real connections, one has a map $B\overline{\ca 
G}_\bbr\simeq E\overline{\ca G}_\bbr\times {\ca A}/\overline{\ca G}_\bbr\rightarrow {\ca A}/\overline{\ca G}_\bbr$, which is a fibration with trivial fibres over the set of generic connections over which $\overline{\ca 
G}_\bbr$ acts freely. One can pull back the Yang-Mills functional to $B\overline{\ca 
G}_\bbr$, and consider it there; the indices stay the same. If one is in the case when degree and rank are coprime (the ``coprime  case"), let us restrict to the sets $B\overline{\ca 
G}_\bbr)_{\leq c}$, $({\ca A}/\overline{\ca G}_\bbr)_{\leq c} $ where the Yang-Mills functional takes values less than $c$, if $c$ is below the  level of the first non-minimal critical point, there are no reducible connections, and then 
\begin{equation}
B\overline{\ca 
G}_\bbr)_{\\leq c}\simeq({\ca A}/\overline{\ca G}_\bbr)_{\leq c} 
\end{equation}
One can then use the results on the homotopy groups that we have computed and transfer them to $({\ca A}/\overline{\ca G}_\bbr)_{\leq c}$, and hence, via the Morse flows,  to the moduli of stable bundles.

\begin{thm}
In the coprime case, for $(n-1)(g-1)>2$, or for $k=0$, $(n,g) = (3,2), (2,3)$, the fundamental groups  and second homotopy groups of the moduli spaces ${\ca M}_\bbr$ of real or quaternionic bundles and ${\ca M}^0_\bbr$ of real or quaternionic bundles of fixed determinant are as follows:

{\footnotesize{
\begin{tabular}{|r||c|c||c|c|}\hline
&&&&\\
&$\pi_1({\ca M}_\bbr )$ &$\pi_2({\ca M}_\bbr )$ & $\pi_1({\ca M}^0_\bbr)$& $ \pi_2({\ca M}^0_\bbr)$ \\ \hline\hline
Real,&$\bbz^{g}\oplus\bbz/2$&$\bbz\oplus \bbz/2$ &$0$,&$\bbz\oplus \bbz/2$\\
 type 0  & ($n$ even)&   ($n$ even) && ($n$ even)\\
 &$\bbz^{g}$($n$ odd)&  $\bbz$ ($n$ odd, $>1$)&&$\bbz$ ($n$ odd, $>1$)\\
\hline
Real,&$\bbz^{g}\oplus (\bbz/2)^{r+1}$&$\bbz\oplus (\bbz/2)^{r+1}$&$(\bbz/2)^{r}$&$\bbz\oplus (\bbz/2)^{r+1}$\\
type I & ($n$ even)& ($n =0(4))$& & ($n =0(4))$\\
& $\bbz^{g}\oplus (\bbz/2)^{r}$ &$\bbz\oplus(\bbz/2)^{r-1}\oplus\bbz/4$&&$\bbz\oplus(\bbz/2)^{r-1}\oplus\bbz/4$\\
&  ($n$ odd)&  ($n =2(4),n>2)$&& ($n =2(4))$\\
&&$\bbz\oplus(\bbz/2)^{r}$&&$\bbz\oplus(\bbz/2)^{r}$\\
&&($n$ odd)&&($n$ odd)\\
&$\bbz^{g+r}\oplus \bbz/2 (n=2) $&$\bbz^{r+1} (n=2) $&$\bbz^{r} (n=2)$&$\bbz^{r+1} (n=2)$\\
\hline
Real,&$\bbz^{g}\oplus (\bbz/2)^{r+1}$&$\bbz\oplus (\bbz/2)^{r+1}$&$(\bbz/2)^{r}$&$\bbz\oplus (\bbz/2)^{r+1}$\\
type II&  ($n$ even)&  ($n =0(4))$&& ($n =0(4))$\\
& $\bbz^{g}\oplus (\bbz/2)^{r}$&  $\bbz\oplus(\bbz/2)^{r-1}\oplus\bbz/4$&&$\bbz\oplus(\bbz/2)^{r-1}\oplus\bbz/4$\\
&  ($n$ odd)&  ($n =2(4)n>2)$&& ($n =2(4)n>2)$\\
&&$\bbz\oplus(\bbz/2)^{r}$&&$\bbz\oplus(\bbz/2)^{r}$\\
&&($n$ odd)&&($n$ odd)\\
&$\bbz^{g+r}\oplus \bbz/2 (n=2) $&$\bbz^{r+1} (n=2) $&$\bbz^{r} (n=2)$&$\bbz^{r+1} (n=2)$\\
\hline
Quat., &$\bbz^{g}\oplus\bbz/2$,&$\bbz\oplus \bbz/2$&$0$&$\bbz\oplus \bbz/2$\\
 type 0 &  ($n$ even)& ($n$ even) && ($n$ even)\\
& $\bbz^{g}$($n$ odd)&  $\bbz$ ($n$ odd) &&$\bbz$ ($n$ odd)\\
\hline
Quat.,&$\bbz^{g}$,&$\bbz\oplus \bbz/2$&$0$&$\bbz\oplus \bbz/2$\\
 type I &&&&\\
(n even) &&&&\\
\hline
 Quat.,&$\bbz^{g}\oplus \bbz/2$&$\bbz\oplus \bbz/2$&$0$&$\bbz\oplus \bbz/2$\\
 type II &&&&\\
(n even) &&&&\\
\hline \end{tabular} }}
\bigskip

For $(n,g)= (3,2), (2,3)$ the fundamental groups are as given in the table above, and the second homotopy groups are quotients of the groups given above. For  $(n,g)= (2,2)$, the fundamental groups are quotients of the groups given in the table above.
\end{thm}

One can also give results for framed bundles, that is the moduli of pairs $(E,t)$, where $E$ is polystable and $t$ is a trivialisation at the base point; here one is quotienting the space of connections by the based gauge group ${\ca G}^0_\bbr$;
this action is always free, and so life is simpler; one has
\begin{equation}
B{\ca G}^0_\bbr \simeq {\ca A}/ {\ca G}^0_\bbr 
\end{equation}
and so then one can transfer our homotopy results even in the non-coprime case, reading off the first and second  homotopy groups of the framed moduli spaces from those of  $B{\ca G}^0_\bbr$, and from those of  $B{\ca SG}^0_\bbr$ in the fixed determinant case, referring to the tables above.

\end{document}